\documentclass[11pt,reqno]{amsart}
\usepackage[utf8]{inputenc}
\usepackage{amsmath}
\usepackage{amsfonts}
\usepackage{amssymb}
\usepackage{amsthm}
\usepackage{hyperref}
\usepackage{multicol}
\usepackage{pictexwd, dcpic}
\usepackage{fancyhdr}
\usepackage{mathrsfs}
\usepackage{latexsym,mathtools}
\usepackage{bm, enumerate}
\usepackage{graphicx}
\usepackage{color,soul}
\usepackage{xcolor}
\usepackage{hyperref}
\usepackage{dsfont}
\usepackage{todonotes}
\usepackage{tikz-cd}
\makeatletter
\@namedef{subjclassname@2020}{%
	\textup{2020} Mathematics Subject Classification}
\makeatother

\newcommand{\beq}{\begin{eqnarray}}
	\newcommand{\eeq}{\end{eqnarray}}
\newcommand{\beqn}{\begin{eqnarray*}}
	\newcommand{\eeqn}{\end{eqnarray*}}
\newcommand{\rar}{\rightarrow}
\newcommand*{\mf}{\mathbf}
\newcommand*{\bs}{\boldsymbol}

\reversemarginpar  

\numberwithin{equation}{section}

\newtheorem{theorem}{\bf Theorem}[section]
\newtheorem{proposition}[theorem]{\bf Proposition}
\newtheorem{lemma}[theorem]{\bf Lemma}

\theoremstyle{remark}
\newtheorem{definition}[theorem]{\bf Definition}
\newtheorem{example}[theorem]{\bf Example}
\newtheorem{Non-example}[theorem]{\bf Non-Example}
\newtheorem{remark}[theorem]{\bf Remark}

\newcommand{\clh}{\mathcal{H}}

\newcommand{\clm}{\mathcal{M}}

\newcommand{\cls}{\mathcal{S}}

\newcommand{\z}{\bm{z}}
\newcommand{\w}{\bm{w}}

\newcommand*{\Ge}{\geqslant}

\newcommand*{\inp}[2]{\langle{#1},\,{#2} \rangle}
\newcommand*{\Le}{\leqslant}

 \newcommand{\ntwo}{\mathbb N_{\Ge 2}}

\theoremstyle{definition}

\theoremstyle{remark}

\hypersetup{
	colorlinks,
	citecolor=blue,
	linkcolor=blue,
	urlcolor=blue}

\begin{document}
\title[Isometric rigidity of complete Pick algebras on monomial varieties]{Isometric rigidity of complete Pick algebras on monomial varieties}

\author[H. Ahmed]{Hamidul Ahmed}
\address{Department of Mathematics, Indian Institute of Technology Bombay,
Powai, Mumbai, 400076, India}
\email{hamidulahmedslk@gmail.com, 22D0785@iitb.ac.in}

\author[B.K. Das]{B. Krishna Das}
\address{Department of Mathematics, Indian Institute of Technology Bombay,
Powai, Mumbai, 400076, India}
\email{bata436@gmail.com, dasb@math.iitb.ac.in}

\author[C. K. Sahu]{Chaman Kumar Sahu}
\address{Department of Mathematics, Indian Institute of Technology Bombay,
Powai, Mumbai, 400076, India}
\email{sahuchaman9@gmail.com, chamanks@math.iitb.ac.in}


\begin{abstract}
We establish a geometric rigidity theorem for a class of monomial varieties in finite-dimensional unit balls. Every ball automorphism mapping one such variety onto another is unitary, and whenever such an equivalence exists, it can also be realized by a permutation of the ambient coordinates. As a consequence, we obtain a complete isometric classification of the multiplier algebras of integrally generated complete Nevanlinna--Pick Dirichlet series kernels in terms of their defining weight and frequency data. 
We also obtain an explicit isometric classification for a class of complete Nevanlinna--Pick kernels on bounded Reinhardt domains.
\end{abstract}
\maketitle
\section{Introduction and Main results}

A recurring theme in mathematics is that a space is determined by a certain algebra of functions on it. For compact Hausdorff spaces $X$ and $Y$, the Gelfand-Kolmogorov theorem \cite{Gelfand} states that the algebras $C(X)$ and $C(Y)$ of continuous functions are isomorphic (or isometrically isomorphic) if and only if $X$ and $Y$ are homeomorphic. In complex analysis, a theorem of L. Bers \cite{Bers} asserts that two domains $V$ and $W$ in $\mathbb C$ are biholomorphic if and only if the corresponding algebras of holomorphic functions $\text{Hol}(V)$ and $\text{Hol}(W)$ are isomorphic. Rudin \cite{Rudin-conf} extended this algebraic characterization to arbitrary open Riemann surfaces. 
In Algebraic geometry, two affine varieties are isomorphic if and only if the algebras of polynomial functions on the varieties are isomorphic \cite[Corollary~I.3.7]{Hartshorne}.  This is by no means an exhaustive list of results in this direction; see for instance \cite{Heins, Kakutani, Nakai, Ozkin, Rudin-bounded} and references therein. 

These results strongly indicate in several contexts that the geometry of the underlying domain is encoded in a suitable function algebra. It is therefore natural to ask to what extent this phenomenon persists for function algebras, such as multiplier algebras of reproducing kernel Hilbert spaces. In this paper, we answer this question for multiplier algebras of \emph{integrally generated complete Nevanlinna Pick (CNP) Dirichlet series kernels}. These kernels are specified by a pair of weight and frequency data, and their associated multiplier varieties are \emph{monomial varieties}. Our main result can be summarized as follows:
\begin{quote}
\emph{Two such multiplier algebras are isometrically isomorphic if and only if their associated monomial varieties coincide up to a permutation automorphism of the ambient coordinates. This is also equivalent to the underlying pairs of weight and frequency data admitting a similar pattern, up to a permutation.}
\end{quote}
These results belong to the same reconstruction tradition as the classical theorems above but concern a non-selfadjoint function algebra with its multiplier norm.
In finite dimension, the general theorem of Davidson, Ramsey, and Shalit \cite{DRS} characterizes isometric isomorphism by equivalence of the associated varieties under automorphisms of the unit ball. Our contribution is to reduce this equivalence to a permutation of coordinates for the monomial varieties considered here. As a result, this yields an easily verifiable classification directly in terms of the weight and frequency data associated with the kernel. The key to our classification is a geometric rigidity theorem with two complementary conclusions. Every automorphism of the unit ball that maps one monomial variety in the class considered here onto another must be unitary. Moreover, whenever two such varieties are unitarily equivalent, their equivalence can also be realized by a permutation of the ambient coordinates. 
In particular, it extends the isometric classification obtained in \cite[Theorem~6.2]{ADS} for a single dependent frequency to arbitrary finite integrally generated frequency data. We now describe the setting and state our results precisely.



The study of Hilbert spaces of Dirichlet series began with the work of Hedenmalm, Lindqvist and Seip \cite{HLS}, who identified the multiplier algebra of the Hardy space of Dirichlet series and, as an application, solved a completeness problem of Beurling \cite{Beurling}. These spaces have since been studied extensively in several contexts in literature; see \cite{AOS, BPS, C-S, GH, MS, Seip} and the monograph \cite{QQ}. Reproducing kernel Hilbert spaces with the complete Nevanlinna--Pick property have also been studied extensively, owing to their rich interpolation theory and connections with operator theory; see \cite{AM1, AM, Quiggin, Mccull92, Hartz, ADP}. In the context of Dirichlet series, the CNP spaces were first studied and characterized by McCarthy and Shalit in \cite{MS}. We follow an equivalent characterization of such kernels obtained in \cite[Theorem~3.2]{ADS}.   
CNP Dirichlet series kernels defined on some right half plane $\mathbb H_\rho=\{s \in \mathbb C : \text{Re}(s) > \rho \}$, $\rho\in\mathbb R$, are precisely those of the form
\beqn
K_{\mf b, \mf n}(s,u) = \frac{1}{1-\sum_{j=1}^d b_jn_j^{-s-\overline{u}}},\quad s, u \in \mathbb H_\rho
\eeqn
for some $d\in\mathbb N \cup \{\infty\}$, $d$-tuples $\mf n=(n_1, n_2, \dots, n_d)$ of distinct positive integers with $n_i\ge 2$, and $\mf b=(b_j)_{j=1}^d \in (0, \infty)^d$. Here $\mathbb N$ is the set of all natural numbers excluding zero. 
We refer to $\mf b$ and $\mf n$ as the \emph{weight and frequency} data, respectively. A CNP Dirichlet series kernel is completely determined by its associated weight and frequency data $\mf b$ and $\mf n$; we therefore denote it by $K_{\mf b,\mf n}$. Note that, in terms of the function
$f_{\mf b, \mf n}:\mathbb H_\rho \rar \mathbb B_d$ defined by
\beq
\label{f-at-s-intro}
f_{\mf b, \mf n}(s) =(\sqrt{b_j} n_j^{-s})_{j=1}^d,\quad s \in \mathbb H_\rho,
\eeq
CNP Dirichlet series kernel $K_{\mf b, \mf n}$ can be written as 
\[
K_{\mf b, \mf n}(s,u) = \frac{1}{1-\sum_{j=1}^d b_jn_j^{-s-\overline{u}}}= \frac{1}{1-\langle f_{\mf b, \mf n} (s), f_{\mf b, \mf n}(u) \rangle},\quad s, u \in \mathbb H_\rho,
\]
that is, $K_{\mf b, \mf n}$ can be realized as a pullback of the Drury-Arveson kernel
\[\mathcal S_d(\z,\w)=\frac{1}{1-\langle \z,\w\rangle}, \quad \bm z,\bm w\in\mathbb B_d,
\]
on the unit ball $\mathbb B_d$ via the map $f_{\mf b, \mf n}$. 

From this point onward, we assume that $d$ is a finite integer with $d\geq2$ and restrict our attention to the following class of CNP Dirichlet series kernels, first considered in \cite{MS}. We adopt the following notation throughout: we use $\mathbb Z$ to denote the set of all integers, $\mathbb Z_+:=\{0,1,2,\ldots\}$ to denote the set of nonnegative integers, for an $n$-tuple of nonzero numbers $\bm \beta=(\beta_1,\ldots,\beta_n)$ and $\bm\alpha=(\alpha_1,\ldots,\alpha_n)\in\mathbb Z^n$, we write
\[
\bm \beta^{\bm\alpha}:=\prod_{i=1}^{n}\beta_i^{\alpha_i}.
\]
\begin{definition}
For a frequency data $\mf n = (n_1,n_2,\dots, n_d)$, $(d\ge 2)$, we say that $\mf n$ has {\it a generating subtuple}  if there exists an integer $1\le d' < d$ such that
\begin{itemize}
     \item [(i)] the set $\{\ln n_{i}: i=1,\dots,d'\}$ is linearly independent over $\mathbb Q$; and 
     \item [(ii)] for each $j= d'+1, \ldots, d$, there exists $\bm\alpha_j \in \mathbb Z_+^{d'}$ 
     such that 
\[
n_j =  (\mf n')^{\bm \alpha_j},
\]
where $\mf n' = (n_1,n_2,\dots, n_{d'})$.
 \end{itemize}
In this case,  $\mf n'$ is called a {\it generating subtuple} of $\mf n$.

Further, given a weight data $\mf b=(b_j)_{j=1}^d$, we say $K_{\mf b, \mf n}$ is an {\em integrally generated CNP Dirichlet series kernel} if its underlying frequency data \(\mf n\) has a generating subtuple. 
\end{definition}
For an integrally generated CNP Dirichlet series kernel $K_{\mf b, \mf n}$, the vectors $\bm \alpha_{j}$ appearing in (ii) are uniquely determined as the set $\{\ln n_{i}: i=1,\dots,d'\}$ is linearly independent over $\mathbb Q$. Also, since all the frequencies are distinct, $|\bm \alpha_j|=\sum_{i=1}^{d'}\alpha_{ij} \Ge 2$ for each $j= d'+1, \ldots, d$ and $\bm \alpha_j$'s are distinct. Let $\mf n=(n_j)_{j=1}^d$ be a finite frequency tuple whose logarithms are linearly dependent over $\mathbb Q$. After reordering its entries, we may choose an integer $1\leq d'<d$ such that $\{\ln n_i:i=1,\ldots,d'\}$ is linearly independent over $\mathbb Q$. Moreover, for each $j=d'+1,\ldots,d$, there exist a positive integer $t_j$ and a vector $\bm\alpha_j\in\mathbb Z^{d'}$ such that
\[
n_j^{t_j}=(\mf n')^{\bm\alpha_j},
\qquad
\mf n'=(n_1,\ldots,n_{d'}).
\]
Thus, among frequency tuples with rationally dependent logarithms, condition~\textup{(i)} can always be satisfied after reordering. The additional requirement lies in condition~\textup{(ii)}: the kernel $K_{\mf b,\mf n}$ is integrally generated if and only if the generating subtuple can be chosen so that $t_j=1$ and $\bm\alpha_j\in\mathbb Z_+^{d'}$ for every $j=d'+1,\ldots,d$.

Our primary object of study are the multiplier algebras associated with CNP Dirichlet series kernels. Let $\clh(K_{\mf b,\mf n})$ denote the reproducing kernel Hilbert space on $\mathbb H_\rho$ associated with kernel $K_{\mf b,\mf n}$. A function $\varphi:\mathbb H_\rho\to\mathbb C$ is called a \emph{multiplier} of $\clh(K_{\mf b,\mf n})$ if the multiplication operator
\[
M_\varphi:\clh(K_{\mf b,\mf n})\to\clh(K_{\mf b,\mf n}),
\qquad f\mapsto\varphi f,
\]
is well-defined and bounded. The collection of all such functions forms a unital Banach algebra, denoted by $\clm(\clh(K_{\mf b,\mf n}))$ and equipped with the multiplier norm
\[
\|\varphi\|_{\clm(\clh(K_{\mf b,\mf n}))}:=\|M_\varphi\|.
\]
Our goal is to classify these non-selfadjoint algebras up to isometric isomorphism. This problem is part of the broader study of the structure and classification of multiplier algebras; see, for example, \cite{DHS, DRS2011, DRS, Hartz_con, Hartz_1, KMS, Mc-1, RH, Sa-Sh, Watt}.

The geometric object central to our classification is the multiplier variety associated with $K_{\mf b,\mf n}$. It is the smallest multiplier variety in $\mathbb B_d$ containing $f_{\mf b,\mf n}(\mathbb H_\rho)$, namely
\beqn
V=\left\{\bm z\in\mathbb B_d:
g(\bm z)=0 \ \text{for every } g\in\mathcal M_d
\ \text{such that } g|_{f_{\mf b,\mf n}(\mathbb H_\rho)}\equiv0
\right\},
\eeqn
where $f_{\mf b, \mf n}$ is as in \eqref{f-at-s-intro}, and $\mathcal M_d$ is the multiplier algebra of the Drury-Arveson space over the unit ball $\mathbb B_d$. This construction is analogous to taking the Zariski closure, with polynomials replaced by the multiplier algebra $\mathcal M_d$. In \cite{ADS}, these multiplier varieties were explicitly described by polynomial equations determined by the arithmetic structure of the associated weight and frequency data; we recall this description in Theorem~\ref{variety-gen}. 
 For integrally generated CNP Dirichlet series kernels, the general description of the associated multiplier variety simplifies considerably. Specifically, Proposition~\ref{variety-form} identifies it as the intersection of the unit ball with the graph of a monomial map. This motivates the following definition.

\begin{definition}
Let $d\geq2$. A set $V\subseteq\mathbb B_d$ is called a \emph{monomial variety} if there exist an integer $1\leq d'<d$, positive coefficients $a_{d'+1},\ldots,a_d$, and distinct multi-indices
\[
\bm\alpha_j=(\alpha_{1j},\ldots,\alpha_{d'j})\in\mathbb Z_+^{d'},
\qquad j=d'+1,\ldots,d,
\]
with $|\bm\alpha_j|\geq2$, such that
\[
V=\mathbb B_d\cap\bigcap_{j=d'+1}^{d}
Z\!\left(z_j-a_j\prod_{i=1}^{d'}z_i^{\alpha_{ij}}\right),
\]
where $Z(p)$ denotes the zero set of a polynomial $p$. We refer to $z_1,\ldots,z_{d'}$ as the \emph{generating variables} and say that $V$ has $d'$ generating variables.
\end{definition}
Throughout this paper, the term \emph{monomial variety} refers exclusively to the class specified in the preceding definition. The condition $|\bm\alpha_j|\geq2$ is motivated by the structure of the multiplier varieties arising from integrally generated CNP Dirichlet series kernels. Indeed, Proposition~\ref{variety-form} shows that the multiplier variety associated with every such kernel $K_{\mf b,\mf n}$ is a monomial variety in this sense.

An isometric isomorphism between the multiplier algebras of two integrally generated CNP Dirichlet series kernels forces their weight and frequency data to have the same length \cite[Proposition~5.2]{ADS}. We therefore consider pairs of kernels $K_{\mf b,\mf n}$ and $K_{\mf c,\mf m}$ whose defining data have a common length $d\geq2$. Let $V,W\subseteq\mathbb B_d$ denote their associated multiplier varieties. The role of multiplier varieties in the present classification stems from the canonical isometric isomorphism 
\[
\mathcal M\bigl(\mathcal H(K_{\mf b,\mf n})\bigr)\cong\mathcal M_V,
\]
where $\mathcal M_V=\{g|_V:g\in\mathcal M_d\}$ is equipped with its multiplier norm; see \cite[Theorem~C]{Sa-Sh}.
Therefore, the algebras $\mathcal M(\mathcal H(K_{\mf b,\mf n}))$ and $\mathcal M(\mathcal H(K_{\mf c,\mf m}))$ are isometrically isomorphic if and only if $\mathcal M_V$ and $\mathcal M_W$ are isometrically isomorphic.
For finite $d$, the theorem of Davidson, Ramsey, and Shalit \cite[Theorems~4.5 and~5.9]{DRS} characterizes the latter condition by the existence of an automorphism of $\mathbb B_d$ mapping $V$ onto $W$. Thus the isometric isomorphism problem for integrally generated CNP Dirichlet series kernels reduces to the following concrete geometric question:\\

\noindent\textbf{Problem.} When are two monomial varieties $V,W\subseteq\mathbb B_d$ equivalent under an automorphism of the unit ball, and what restrictions does their monomial structure impose on an automorphism realizing this equivalence?

Our answer establishes rigidity both for the automorphisms realizing an equivalence and for the equivalence classes of the varieties themselves. The former must be unitary, while the latter are determined by coordinate permutations. The following theorem makes both assertions precise. The assumption that $V$ and $W$ have the same number of generating variables entails no loss of generality, since this number equals their complex dimension and is preserved by any biholomorphism. The notation $\mathcal U(n)$ denotes the set of all unitary matrices in $\mathcal B(\mathbb C^n)$.

\medskip
\noindent\textbf{Theorem~A}(Theorem~\ref{thm:existense of permutation} below)
\textit{Let $1\leq d'<d$, and let $V,W\subseteq\mathbb B_d$ be monomial varieties with $d'$ generating variables. Suppose that $\Phi$ is an automorphism of $\mathbb B_d$ satisfying $\Phi(V)=W$.
Then $\Phi$ is unitary: 
there exists $U\in\mathcal U(d)$ such that
\[
\Phi(\bs z)=U\bs z,\qquad \bs z\in\mathbb B_d.
\]
Moreover, there exists a permutation matrix $S\in\mathcal U(d)$ such that $S(V)=W$.}
\medskip

The first conclusion in Theorem A applies to every ball automorphism mapping $V$ onto $W$, whereas the second asserts the existence of a coordinate permutation realizing the same equivalence. The original automorphism $\Phi$ need not itself be a permutation. Consequently, the orbit of a monomial variety under $\operatorname{Aut}(\mathbb B_d)$ contains only finitely many distinct monomial varieties, each obtained by permuting the ambient coordinates. The proof of Theorem~A combines complex geometry with compact Lie group theory and proceeds in two steps. We first associate to each monomial variety a canonical bounded Reinhardt domain that is biholomorphic to the variety. The domain is strongly pseudoconvex with smooth boundary. This model allows us to apply the Wong--Rosay theorem, a result of D\textquotesingle Angelo \cite{DAngelo}, and results of Sunada on Reinhardt domains \cite{Sunada} to reduce equivalence under ball automorphisms to unitary equivalence. We then use the conjugacy of maximal tori in compact connected Lie groups to show that this unitary equivalence can be realized by a permutation of the ambient coordinates.

As an application of Theorem A, we obtain a complete classification, up to isometric isomorphism, of the multiplier algebras of integrally generated CNP Dirichlet series kernels. This classification captures both geometric and arithmetic rigidity: the isometric isomorphism class of the multiplier algebra determines the associated monomial variety up to a coordinate permutation and the defining weight and frequency data up to similarity of patterns and permutation. To state the result precisely, we first recall the necessary notation and terminology.
 Given a $d$-tuple of frequency data $\mf n = (n_j)_{j=1}^{d}$, define 
 \[
 \begin{array}{rl}
 &\mathcal J(\mf n):= \{ J\subseteq \{1, 2, \ldots, d\}: \{\ln n_j: j \in J\} \text{ is linearly dependent over}\ \mathbb Q \text{ and all of its} \\  & \hspace{3cm} \text{ proper subsets are linearly independent over}\ \mathbb Q\}.
 \end{array}
 \]
In \cite[Proposition~4.4]{ADS}, it is shown that corresponding to each $J\in \mathcal J(\mf n)$ there exist a unique partition of $J$ into nonempty sets $\{J_1, J_2\}$ and a unique $|J|$-tuple of positive integers $(\beta_j)_{j\in J}$ with $\text{gcd}(\beta_j)_{j\in J}=1$ such that
\begin{align*}
 \prod\limits_{i\in J_1}n_i^{\beta_i}=\prod\limits_{i\in J_2}n_i^{\beta_i}.
\end{align*}
This enables us to define the notion of pairs of weight and frequency data admitting similar pattern. 
\begin{definition}
Let $(\mf b, \mf n)$ and $(\mf c, \mf m)$ be two pairs of weight and frequency data of length $d\ge 2$.
We say that $(\mf b, \mf n)$ and $(\mf c, \mf m)$ admit a {\it similar pattern} if:
\begin{itemize}
\item[(i)] $\mathcal J(\mf n) = \mathcal J(\mf m)$, and
\item[(ii)] for each $J \in \mathcal J(\mf n)$, the unique partition $\{J_1, J_2\}$ and the unique $|J|$-tuple $(\beta_i)_{i \in J}$ associated with $J$ remain the same whether $J$ is viewed as an element of $\mathcal J(\mf n)$ or $\mathcal J(\mf m)$, and the following identity holds:
\[
\prod\limits_{i \in J_1}
 		c_i^{\beta_i}\prod\limits_{i \in J_2}  
   b_{i}^{\beta_i}=\prod\limits_{i \in J_2}c_i^{\beta_i}\prod\limits_{i \in J_1}b_i^{\beta_i}.
\]
\end{itemize}
\end{definition}
In \cite[Theorem~5.6]{ADS}, it was shown that if the pairs $(\mf b, \mf n)$ and $(\mf c, \mf m)$ admit a similar pattern, then the multiplier varieties associated with  CNP Dirichlet series kernels $K_{\mf b, \mf n}$ and $K_{\mf c, \mf m}$ are equal and the corresponding multiplier algebras are isometrically isomorphic. Interestingly, for integrally generated CNP Dirichlet series kernels, it turns out that the converse also holds, up to a permutation. For a permutation $\sigma$ of $\{1,\ldots, d\}$ and a $d$-tuple $\bm \alpha=(\alpha_1,\alpha_2,\ldots, \alpha_d)$ of numbers, we denote
\[\sigma(\bm \alpha)=(\alpha_{\sigma(1)}, \alpha_{\sigma(2)}, \dots, \alpha_{\sigma(d)}).\]   \\

\noindent 
\textbf{Theorem~B}(Theorem~\ref{rigidity-Isometric} below).
    \textit{Let $K_{\mf b, \mf n}$ and $K_{\mf c, \mf m}$ be two integrally generated CNP Dirichlet series kernels associated with the pairs $(\mf b, \mf n)$ and $(\mf c,\mf m)$ of weight and frequency data of length $d$. Then the following statements are equivalent.
    \begin{itemize}
\item [(i)] $\clm(\clh(K_{\mf b,\mf n}))$ and $\clm(\clh(K_{\mf c,\mf m}))$ are isometrically isomorphic.
\item [(ii)] There exists a permutation matrix $U\in\mathcal U(d)$ such that
        \[
        U(V)=W,
        \]
where $V$ and $W$ are the multiplier varieties associated with $K_{\mf b, \mf n}$ and $K_{\mf c, \mf m}$, respectively.
 \item [(iii)] There exists a permutation $\sigma$ of $\{1,2,\ldots,d\}$ such that $(\sigma (\mf b),\sigma (\mf n))$ and $(\mf c, \mf m)$ admit a similar pattern.
    \end{itemize}} 
\medskip

In particular, for monomial varieties $V,W\subseteq\mathbb B_d$ in the class considered here, we obtain
\[
\mathcal M_V\overset{\mathrm{isometric}}{\cong}\mathcal M_W
\quad\Longleftrightarrow\quad
U(V)=W
\quad\text{for some permutation matrix } U.
\]
Thus the multiplier algebra, equipped with its multiplier norm, determines the embedded monomial variety up to a permutation of the ambient coordinates. Condition~\textup{(iii)} expresses this geometric classification directly in terms of the defining weight and frequency data, yielding an explicit and verifiable arithmetic criterion for isometric isomorphism.
It identifies exactly the information retained by the isometric multiplier algebra: the coordinate-permutation class of the embedded variety, or equivalently the similarity-pattern class of the defining data up to permutation. In particular, the numerical values of the frequencies and the individual weights need not be recovered; the invariant is their pattern of relations and associated weight ratios. The essential geometric implication is \textup{(i)}$\Rightarrow$\textup{(ii)}. The theorem of Davidson, Ramsey, and Shalit supplies a ball automorphism between the associated varieties. These varieties being monomial varieties, Theorem~A shows that this automorphism must be unitary and that a coordinate permutation also maps one variety onto the other. The connection between the defining equations and the weight and frequency data yields the remaining classification.

Theorem~B extends the isometric classification obtained in \cite[Theorem~6.2]{ADS}.
The proof in \cite{ADS} passes from the associated multiplier varieties to the underlying Reinhardt domains. Then uses case-by-case calculations involving their defining polynomials and the biholomorphisms between them to show that the two varieties are related by a permutation of the ambient coordinates. As the number of dependent frequencies increases, these calculations become substantially more complicated and do not readily yield a unified classification. Theorem~A replaces this case analysis with a structural rigidity argument, yielding a uniform proof of the general classification and a conceptual alternative to the earlier proof in the single-dependent-frequency case.

Finally, we apply our classification to CNP kernels on bounded Reinhardt domains in several complex variables. Specifically, consider
\[
\mathcal S_{\mf b,\bm\alpha}(\bs z,\bs w)
=
\frac{1}{
1-\sum_{i=1}^{d'}b_i z_i\overline{w_i}
-\sum_{j=d'+1}^{d}b_j \bs z^{\bm\alpha_j}
\overline{\bs w}^{\bm\alpha_j}
},
\]
defined on
\[
\Omega_{\mf b,\bm\alpha}
=
\Big\{
\bs z\in\mathbb C^{d'}:
\sum_{i=1}^{d'}b_i|z_i|^2
+\sum_{j=d'+1}^{d}b_j|\bs z^{\bm\alpha_j}|^2<1
\Big\},
\]
where the weights $b_i$'s are positive and the multi-indices $\bm\alpha_j\in\mathbb Z_+^{d'}$ are pairwise distinct. We show that the multiplier algebras associated with $\cls_{\mf b,\bm\alpha}$ and $\cls_{\mf c,\bm\beta}$ are isometrically isomorphic if and only if there exist permutations $\sigma$ of $\{1,\ldots,d'\}$ and $\sigma'$ of $\{d'+1,\ldots,d\}$ such that
\[
\sigma(\bm\alpha_{\sigma'(j)})=\bm\beta_j,
\qquad
(\mf c')^{\bm\beta_j}b_{\sigma'(j)}
=
\sigma(\mf b')^{\bm\beta_j}c_j,
\qquad j=d'+1,\ldots,d,
\]
where $\mf b'=(b_1,\ldots,b_{d'})$ and $\mf c'=(c_1,\ldots,c_{d'})$. This gives an explicit classification in terms of the monomial exponents and weights. The result follows by transferring Theorem~B through the correspondence established in \cite[Section~7]{ADS}, which preserves the associated multiplier varieties, and extends the single-monomial case of \cite[Theorem~7.3]{ADS}; see Theorem~\ref{rigidity-isometric-classical}.

The paper is organized as follows. In the next section, we prove the geometric rigidity theorem, Theorem~A. In the final section, we establish the classification theorem, Theorem~B, and apply it to multiplier algebras of CNP kernels on bounded Reinhardt domains.

\section{A geometric rigidity result on monomial varieties}
For the present purpose, a monomial variety $V \subseteq \mathbb B_d$ with $d'$ generating variables is part of a variety in $\mathbb C^d$ of the form 
\begin{equation}\label{monomial variety}
V=\big\{\bs z\in \mathbb B_d:z_j=B_j \prod_{i=1}^{d'}z_i^{\alpha_{ij}},\ j = d'+1, \ldots, d\big\},
\end{equation}
where $B_j\in (0, \infty)$, and  $\bm \alpha_j =(\alpha_{1j},\alpha_{2j},\dots,\alpha_{d'j})\in \mathbb Z_+^{d'}$ for all $j = d'+1, \ldots, d$ with $|\bm \alpha_j| \Ge 2$ are distinct. Each of these monomial varieties is also a graph over a domain 
\begin{equation} \label{graph-domain}
\widetilde{V}:=\Big\{\bs z\in \mathbb C^{d'}: \sum_{j=1}^{d'}|z_j|^2+\sum_{j=d'+1}^d|B_j|^2|\bs z^{\bm \alpha_j}|^2 < 1\Big\}.
\end{equation}
Indeed, the monomial variety can also be expressed as 
\[
V=\{(\bm z, f(\bm z)): \bm z\in \widetilde{V}\},\ \text{where } 
f(\bm z)= (B_{d'+1}\bm z^{\bm\alpha_{d'+1}},\dots, B_d\bm z^{\bm\alpha_d}), \quad \bm z\in \mathbb C^{d'}.
\]
For such a reason, it is also appropriate to call these monomial varieties as monomial graph varieties. Being a graph over $\widetilde{V}$, the variety $V$ is biholomorphic to $\widetilde{V}$, via the canonical inclusion
\begin{equation}\label{F-map}
F_{\text v}(\bs z)=(z_1,\ldots,z_{d'},B_{d'+1}\bs z^{\bm \alpha_{d'+1}},\ldots,B_d\bs z^{\bm \alpha_d}),\quad \bs z=(z_1,\ldots,z_{d'})\in\widetilde{V}
\end{equation}
Therefore, these associated domains will be very crucial in our study of monomial varieties. Before proceeding further, we record various properties of $\widetilde{V}$.  The domain $\widetilde{V}$ is obviously a complete Reinhardt domain in $\mathbb C^{d'}$. 
We refer to $\widetilde{V}$ as the Reinhardt domain associated to $V$.

 Given  a domain $\Omega \subseteq  \mathbb C^n$, we sometimes use the identification $\mathbb C^n\cong \mathbb R^{2n}$ via $z_j=x_j+iy_j$.
\begin{definition} 
Let $n \in \mathbb N$ and $k \in \mathbb N \cup \{\infty\}$. A domain $\Omega\subset\mathbb C^n$ is said to have {\it $C^k$ boundary} at a point $p\in \partial\Omega$ if there exist an open set $U\subset\mathbb R^{2n}$ containing $p$ and a $k$ times continuously differentiable function $\rho:U\to\mathbb R$ such that
    \begin{itemize}
        \item[(i)] $U\cap\Omega=\{\bs z\in U:\rho(\bs z)<0\}$,
        \item[(ii)] the gradient $\nabla\rho(\bs z)\neq 0$, for all $\bs z\in U$.
    \end{itemize} 
 We say $\Omega$ has $C^k$ boundary if it has $C^k$ boundary for every $p\in\partial\Omega$. The function $\rho$ is called {\it a local defining function} for $\Omega$ at $p$. If $U$ is a neighbourhood of $\partial\Omega$ and $\rho:U\to\mathbb R$ a $k$ times continuously differentiable function on $U$ satisfying (i) and (ii) above, then $\rho$ is called {\it a defining function for $\Omega$}.
\end{definition}
It is well known that if $\Omega \subseteq \mathbb C^n$ is a bounded domain with $C^k$ boundary, then there is a defining function for $\Omega$, see \cite[Lemma~2.2.6]{Range}. We also need the notion of complex tangent space to define pseudoconvexity of a domain.
\begin{definition}
    Let $\Omega$ be a domain in $\mathbb C^n$, and let $\rho$ is a local defining function for $\Omega$ at $p\in \partial\Omega$. The {\it complex tangent space} to $\partial\Omega$ at $p$ is defined by
    \[
    T_p^\mathbb C(\partial\Omega):=\Big\{(w_1,w_2,\dots,w_n)\in\mathbb C^n:\sum_{j=1}^nw_j\frac{\partial\rho}{\partial z_j}(p)=0\Big\}
    \]
\end{definition}

We next present the notion of pseudoconvexity of a domain in $\mathbb C^n$, see \cite[Chapter 2]{Range} for more details.
\begin{definition}
    A domain $\Omega$ in $\mathbb C^n$ with $C^2$ boundary is said to be {\it pseudoconvex or Levi pseudoconvex}  if for every $p\in\partial\Omega$, there exists a $C^2$ local defining function $\rho$ for $\Omega$ at $p$ such that the Levi form
    \[
    L_p(\rho;\bs w):=\sum_{j,k=1}^nw_j\overline{w}_k\frac{\partial^2\rho}{\partial z_j\partial \overline{z}_k}(p)\Ge 0\quad\text{for all}\ \bs w=(w_1,w_2,\dots,w_n)\in T_p^\mathbb C(\partial \Omega).
    \]
    The domain $\Omega$ is called {\it strongly pseudoconvex} if $ L_p(\rho;\bs w)>0$ for all $\bs w\in T_p^\mathbb C(\partial \Omega) \setminus \{0\}$.
\end{definition}


 In the following lemma, we record some further properties of the domain $\widetilde{V}$. We fix the collection $\{e_i\}_{i=1}^n$ to denote the standard orthonormal basis of $\mathbb C^n$.

\begin{lemma}
Let $V \subseteq \mathbb B_d$ be a monomial variety with $d'$ generating variables, and let $\widetilde{V}$ be the Reinhardt domain associated to $V$, as in \eqref{graph-domain}. Then, $\widetilde{V} \subseteq \mathbb C^{d'}$ is a strongly pseudoconvex Reinhardt domain with smooth boundary.
    %
\end{lemma}
\begin{proof}
Let $V$ be a monomial variety as in ~\eqref{monomial variety}. It then follows from ~\eqref{graph-domain} that 
$\widetilde{V}=\{\bs z\in \mathbb C^{d'}: \rho(\bs z)<0\},
$
where
$\rho:\mathbb C^{d'}\to\mathbb R$ is given by
$$\rho(\bs z) =-1+\sum_{j=1}^{d'}|z_j|^2+\sum_{j=d'+1}^d|B_j|^2|\bs z^{\bm \alpha_j}|^2, \quad \bs z = (z_1, \ldots, z_{d'}) \in \mathbb C^{d'}.$$
Since 
\(
\rho(z_1,z_2,\ldots,z_{d'})= \rho(e^{i\theta_1}z_1,e^{i\theta_2}z_2,\ldots,e^{i\theta_{d'}}z_{d'}) ~\text{for all}\ \theta_i\in\mathbb R,
\)
it follows that $\widetilde{V}$ is a Reinhardt domain. Now, we show that $\widetilde{V}$ has smooth boundary. To this end, by identifying  $z_j=x_j+iy_j$ and viewing $\widetilde{V}$ as a domain in $\mathbb R^{2d'}$, if we take $U=\mathbb R^{2d'}\setminus\{0\}$, then 
\[
U\cap\widetilde{V}=\{\bs z\in U: \rho(\bs z)<0\}
\]
and for all $i = 1, 2, \ldots, d'$,
\[
\frac{\partial \rho}{\partial x_i}=2x_i\Big(1+\sum_{j=d'+1}^d{\alpha_{ij}}|B_j|^2|\bs z^{\bm \alpha_j-e_i}|^2\Big),\quad \frac{\partial \rho}{\partial y_i}=2y_i\Big(1+\sum_{j=d'+1}^d\alpha_{ij}|B_j|^2|\bs z^{\bm \alpha_j-e_i}|^2\Big),
\]
with the convention that $\bs z^{\bm \alpha} = 0$ if at least one of the coordinates of $\bm \alpha$ is negative.
This shows that the gradient $\nabla\rho(\bs z)=0$ if and only if $\bs z=0$. Since $\rho$ is a smooth function, we then conclude that $\rho$ is a defining function for $\widetilde{V}$; equivalently, $\widetilde{V}$ has smooth boundary. 
Moreover, for $\bs p\in\partial\widetilde{V}$
\[
\Big(\frac{\partial^2}{\partial z_i\partial \overline{z}_j}(-1+\sum_{k=1}^{d'}|z_k|^2)(\bs p)\Big)_{i,j=1}^{d'}=  I_{d'}.
\]
Also, for all $k=d'+1, \ldots, d$, 
\[
\Big(\frac{\partial^2}{\partial z_i\partial \overline{z}_j}(|B_k|^2|\bs z^{\bm \alpha_k}|^2)(\bs p)\Big)_{i,j=1}^{d'}=|B_k|^2\Big(\langle \alpha_{ik}\bs p^{ \alpha_k-e_i}, \alpha_{jk} \bs p^{\bm \alpha_k-e_j} \rangle\Big)_{i,j=1}^{d'}
\]
is a positive semi-definite matrix. Therefore,
\[
\Big(\frac{\partial^2 \rho}{\partial z_i\partial \overline{z}_j }(\bs p)\Big)_{i,j=1}^{d'}= I_{d'}+\sum_{k=d'+1}^d\Big(\frac{\partial^2}{\partial z_i\partial \overline{z}_j}(|B_k|^2|\bs z^{\bm \alpha_k}|^2) (\bs p)\Big)_{i,j=1}^{d'}
\]
is positive definite.
Hence, the Levi form $ L_{\bs p}(\rho;\bs w)>0$ for all nonzero vector $\bs w$ in the complex tangent space $T_{\bs p}^\mathbb C(\partial \widetilde{V})$, proving that $\widetilde{V}$ is strongly pseudoconvex. This completes the proof.
\end{proof}

Let $\Omega$ be a bounded strongly pseudoconvex Reinhardt domain containing the origin. If $\Omega$ is biholomorphic to the unit ball in $\mathbb C^d$, then the group of automorphisms $\text{Aut}(\Omega)$ is not compact and it acts on $\Omega$ transitively, as this is the case for the unit ball. On the other hand, if $\Omega$ is not biholomorphic to the unit ball, then by the Wong-Rosay Theorem (see, \cite[Main Theorem]{Wong}), $\text{Aut}(\Omega)$ is compact, and in this case we show below that the automorphism group has a unique fixed point namely the origin. 
\begin{lemma}\label{lemma:fixed point is origin}
Let $\Omega$ be a bounded strongly pseudoconvex Reinhardt domain in $\mathbb C^d$ containing the origin. If $\Omega$ is not biholomorphic to $\mathbb B_d$, then the origin is the unique fixed point for $\text{Aut}(\Omega)$. 
\end{lemma}
\begin{proof}
As noted earlier, by the Wong-Rosay Theorem, $\text{Aut}(\Omega)$ is compact, and so is the identity component of $\text{Aut}(\Omega)$, denoted by $\text{Aut}^\circ(\Omega)$. 
Let 
\[
\mathcal G_0 : = \{g(0) : g \in \text{Aut}^\circ(\Omega)\}
\]
be the $\text{Aut}^\circ(\Omega)$-orbit of the origin. Then $\mathcal G_0$ is compact. Moreover, since $\Omega$ is a bounded Reinhardt domain containing the origin, an application of the classification result \cite[Corollary 2]{Sunada} of T. Sunada, compactness of $\mathcal G_0$ forces that  
    $
    \mathcal G_0 = \{0\}.
    $
Consequently, origin is a fixed point of $\text{Aut}^\circ(\Omega)$. Note that the origin is the only fixed point of the torus group 
\begin{equation*}
    R_\Theta:=\text{diag}(e^{i\theta_1}, e^{i\theta_2}, \ldots,e^{i\theta_d})\in\text{Aut}^\circ(\Omega) \quad \text{for all}\ \Theta=(\theta_1,\theta_2,\dots,\theta_d)\in \mathbb R^d.
    \end{equation*}
Since $\text{Aut}^\circ(\Omega)$ is a normal subgroup of $\text{Aut}(\Omega)$, then for any $g\in \text{Aut}(\Omega)$ and for any element $R_{\Theta}$ of the torus group we have 
\[
g^{-1}\circ R_\Theta\circ g\in\text{Aut}^\circ(\Omega), \text{ and hence, } g^{-1}\circ R_\Theta\circ g(0)=0.
\]
This shows that $g(0)$ is a fixed point of $R_{\Theta}$ for all $\Theta\in \mathbb R^d$. Thus, by the unique fixed point property of the torus group, $g(0)=0$ for all $g\in \text{Aut}(\Omega)$. Since any fixed point of $\text{Aut}(\Omega)$ is also a fixed point of the torus group of automorphisms, the unique fixed point property also follows. This completes the proof. 
\end{proof}

\subsection{Rigidity of automorphisms which takes one monomial variety onto another:} Isometric isomorphisms of multiplier algebras are determined by biholomorphisms between their multiplier varieties induced by automorphisms of the unit ball. Therefore, to address the isometric isomorphism problem, one needs to determine precisely which monomial varieties are biholomorphic via automorphisms of the unit ball. The automorphism group of the unit ball is completely classified, and it turns out that not every automorphism of the unit ball can induce a biholomorphism between monomial varieties. The structure of monomial varieties imposes a certain rigidity, which we investigate next. Specifically, we show that if an automorphism of the unit ball maps one monomial variety onto another, then the automorphism must necessarily be linear.   

\begin{proposition}\label{prop:unitary-representation}
Let $V, W \subseteq \mathbb B_d$ be two monomial varieties with $d'$ generating variables. 
If $\Phi$ is  an automorphism of $\mathbb B_d$ such that $\Phi(V)=W$, then $\Phi$ is a unitary map on $\mathbb C^d$. 

Furthermore, $\Phi$ has the block matrix form
 \[
 \Phi=\begin{bmatrix}
     A & 0\\
     0 & D
 \end{bmatrix},
 \]
 with respect to the decomposition $\mathbb C^{d'}\oplus \mathbb C^{d-d'}$.
\end{proposition}
\begin{proof}
Let $\Phi$ be an automorphism of $\mathbb B_d$ such that $\Phi(V)=W$. Then the map  
\begin{equation} \label{biholomorphism-varphi}
     \varphi:=F_{\text{w}}^{-1}\circ\Phi\circ F_{\text{v}} 
\end{equation}
is a biholomorphism between the associated bounded strongly pseudoconvex Reinhardt domains $\widetilde{V}$ and $\widetilde{W}$ in $\mathbb C^{d'}$, where $F_{\text{v}}: \widetilde{V} \to V$ and $F_{\text{w}}: \widetilde{W}\to W$ are the biholomorphism as defined in \eqref{F-map}. 
We show that $\Phi$ is unitary in each of the following two exhaustive cases.  

\textbf{Case I.} Suppose that the domain $\widetilde{V}$ is not biholomorphic to $\mathbb B_{d'}$. Then, by \eqref{biholomorphism-varphi}, $\widetilde{W}$ is also not biholomorphic to $\mathbb B_{d'}$.
Since $\varphi^{-1} \circ R_\Theta\circ \varphi\in \text{Aut}(\widetilde{V}) $,  then by Lemma~\ref{lemma:fixed point is origin},
     \[\varphi^{-1} \circ R_\Theta\circ \varphi(0)=0,\]
for all $\Theta=(\theta_1,\theta_2,\dots,\theta_{d'})\in \mathbb R^{d'}$, where 
$
 R_\Theta:=\text{diag}(e^{i\theta_1},e^{i\theta_2},\ldots,e^{i\theta_{d'}})\in\text{Aut}(\widetilde{W}).
$
Since origin is the unique point in $\widetilde{W}$ which is fixed by all $R_{\Theta}$, we have $\varphi(0)=0$. Therefore, 
 $$\Phi(0) = F_{\text{w}} \circ \varphi\circ F_{\text{v}}^{-1}(0)=F_{\text{w}}(\varphi(0))=F_{\text{w}}(0)=0.$$  Since $\Phi(0)=0$, it then follows from the description of the automorphism group of $\mathbb B_d$ (see \cite[Theorem~2.2.5]{Rudin}) that $\Phi$ is a unitary map.
     
\textbf{Case II.} Suppose that $\widetilde{V}$ is biholomorphic to $\mathbb B_{d'}$. Then, by \eqref{biholomorphism-varphi}, $\widetilde{W}$ is also biholomorphic to $\mathbb B_{d'}$. Let $\phi_{\text{v}}:\mathbb B_{d'}\to\widetilde{V}$ and $\phi_{\text{w}}:\mathbb B_{d'}\to\widetilde{W}$ be biholomorphisms such that $\phi_{\text{v}}(0)=0=\phi_{\text{w}}(0)$, and the components of both $\phi_{\text{v}}$ and $\phi_{\text{w}}$ are linear. The existence of such biholomorphisms can be ensured by a result of T. Sunada (see, \cite[Main Theorem]{Sunada}).  Set 
\[
f_{\text{v}}:=F_{\text{v}}\circ \phi_{\text{v}}:\mathbb B_{d'}\to V\subseteq\mathbb B_d\ \text { and } f_{\text{w}}:=F_{\text{w}}\circ \phi_{\text{w}}:\mathbb B_{d'}\to W\subseteq\mathbb B_d.
\]
Then $f_{\text v}$ and $f_{\text w}$ are proper polynomial maps with $f_{\text{v}}(0)=f_{\text{w}}(0)=0$. Moreover, the automorphism $\Phi$ and $\Psi$ intertwines the proper maps $f_{\text{v}}$ and $f_{\text{w}}$, where 
\[
\Psi:=\phi_{\text{w}}^{-1}\circ\varphi\circ\phi_{\text{v}}:\mathbb B_{d'}\to\mathbb B_{d'}.
\]
That is, \[
\Phi\circ f_{\text{v}}=\Phi\circ F_{\text{v}}\circ\phi_{\text{v}}= f_{\text{w}}\circ\phi_{\text{w}}^{-1}\circ F_{\text{w}}^{-1}\circ\Phi\circ F_{\text{v}}\circ\phi_{\text{v}}=f_{\text{w}}\circ \phi_{\text{w}}^{-1}\circ\varphi\circ\phi_{\text{v}}=f_{\text{w}}\circ\Psi.
\]
Then it follows from the rigidity result \cite[Theorem 9]{DAngelo} that both the automorphisms $\Phi$  and $\Psi$ are unitary map. This proves the first part of the proposition.


For the moreover part, let $U$ be the matrix representation of $\Phi$. Since $U(V)=W$ and $U(0)=0$, $U(T_0V)=T_0W$ where $T_0V$ and $T_0W$ are the tangent spaces at the origin. The proof follows once we establish that 
\[
T_0V=\text{span}\{e_1,e_2,\ldots,e_{d'}\}=T_0W.
\]
 By the definition of monomial variety, 
$$V=\mathbb B_d\cap\bigcap_{j=d'+1}^dZ(p_{_j}),$$
where $p_{_j}(\bs z)=z_j-B_j \prod_{i=1}^{d'} z_i^{\alpha_{ij}}$ for each $j = d'+1, \ldots, d$ with $B_j>0$ and  $\sum_{i=1}^{d'}\alpha_{ij} \geq 2$. Then, the tangent space $T_0V$ is the kernel of the $(d-d') \times d$ matrix given by
\[
\begin{bmatrix}
    \frac{\partial p_{_{d'+1}}}{\partial z_1}(0) &\frac{\partial p_{_{d'+1}}}{\partial z_2}(0)&\cdots&\frac{\partial p_{_{d'+1}}}{\partial z_d}(0)\\
    \frac{\partial p_{_{d'+2}}}{\partial z_1}(0) &\frac{\partial p_{_{d'+2}}}{\partial z_2}(0)&\cdots&\frac{\partial p_{_{d'+2}}}{\partial z_d}(0)\\
     \vdots & \vdots &\ddots &\vdots\\
    \frac{\partial p_{_d}}{\partial z_1}(0) &\frac{\partial p_{_d}}{\partial z_2}(0)&\cdots&\frac{\partial p_{_d}}{\partial z_d}(0)
\end{bmatrix}.
\]
Since $\sum_{i=1}^{d'}\alpha_{ij} \geq 2$ for all $j$, it follows that $\frac{\partial p_{_i}}{\partial z_j}(0)=\delta_{ij}$ for all $i,j$, and hence 
 \(
 T_0V=\text{span}\{e_1,e_2,\ldots,e_{d'}\}.
 \)
The similar argument also yields that  
 \(
  T_0W=\text{span}\{e_1,e_2,\ldots,e_{d'}\}.
 \)
 This completes the proof.
\end{proof}

We continue our study of rigidity properties of automorphisms of the unit ball associated with monomial varieties. In the preceding discussion, we considered the rigidity arising from automorphisms that maps one monomial variety to another. Here, we investigate a stronger form of rigidity under the additional assumption that the varieties are related by a unitary automorphism. We show that such a unitary equivalence necessarily induces a permutation automorphism. To establish this result, we first introduce some terminology from Lie group theory (see \cite[Chapter IV]{Br_Dieck}) and prove a lemma.

For a function $F:\mathbb C^n\to\mathbb C$, the \emph{symmetry group} of $F$ is defined by
\[
\mathcal G
:=
\{U\in \mathcal U(n): F(U(\bs z))=F(\bs z)
\text{ for all } \bs z\in\mathbb C^n\}.
\]
The elements of $\mathcal G$ are called \emph{unitary invariant transformations} of $F$. For a subgroup $G$ of $\mathcal U(n)$, we say that $F$ is \emph{$G$-invariant} if $G\subseteq\mathcal G$.
In particular, the functions that are invariant under the torus group
\[
\mathbb T^n
=
\left\{
R_{\Theta}
=
\operatorname{diag}
\left(
e^{i\theta_1},e^{i\theta_2},\dots,e^{i\theta_n}
\right)
:
\Theta=(\theta_1,\theta_2,\dots,\theta_n)\in\mathbb R^n
\right\}
\]
are called \emph{radial} functions.

Recall that, for a Lie group $G$, a subgroup $H\subseteq G$ is called a \emph{torus} if $H$ is a Lie group isomorphic to $\mathbb T^k$ for some $k$. Equivalently, a torus is a compact, connected, abelian subgroup of a Lie group. A torus $H$ in $G$ is called a \emph{maximal torus} if it is not properly contained in any other torus of $G$.
The following lemma is the key step in obtaining a permutation automorphism from a unitary automorphism.

\begin{lemma}\label{invariant by permutation}
Let $P:\mathbb C^n\to \mathbb C$ be a continuous radial function. 
If $P$ is an $U^{-1}\mathbb T^nU$-invariant function for some unitary $U\in \mathcal U(n)$ then there exists a permutation matrix $S\in \mathcal U(n)$ such that $P$ is $S^{-1}\mathbb T^nS$-invariant. Moreover, $U$ and $S$ are related by \[UT=R_{\Theta}S\]
for some $R_{\Theta}\in\mathbb T^{n}$ and for some unitary invariant transformation $T\in \mathcal U(n)$ of $P$. 
\end{lemma}
\begin{proof}
Suppose that $P$ is $U^{-1}\mathbb T^nU$-invariant for some unitary $U$. Define the symmetry group of $P$ as 
 \[
 G:=\{M\in \mathcal U(n):P(M\bs z)=P(\bs z) ~ \text{for all $\bs z\in \mathbb C^{n}$}\}.
 \]
Then $G$ is a closed subgroup of $\mathcal U(n)$, as $P$ is continuous, and hence $G$ is a compact Lie group. The identity component $G^{\circ}$ of $G$ is therefore a connected compact Lie group. By the hypothesis,
\[
\mathbb T^n\cup (U^{-1}\mathbb T^n U
)\subseteq G.\]
Also since both $\mathbb T^n$ and $U^{-1}\mathbb T^n U$ are connected, and $I_n\in \mathbb T^n\cap (U^{-1}\mathbb T^n U)$, we have 
\[
\mathbb T^n\cup (U^{-1}\mathbb T^n U
)\subseteq G^{\circ}.
\]
It is well known that the torus group $\mathbb T^n$ is a maximal torus in $\mathcal U(n)$ (see \cite[Proposition 3.1, Chapter 4]{Br_Dieck}), and hence so is its conjugate $U^{-1}\mathbb T^n U$. Thus, both $\mathbb T^n$ and $U^{-1}\mathbb T^n U$ are maximal torus group of $G^{\circ}$. Since any two maximal torus are conjugate of each other
(\cite[Theorem 1.6, Chapter 4]{Br_Dieck}), there exists a $T\in G^\circ$ such that 
 \begin{equation}\label{invariant for w}
\mathbb T^n=T^{-1}U^{-1}\mathbb T^n U T.
 \end{equation}
 Let $M=UT$. By the above identity, for any $R_{\Theta}\in\mathbb T^n$ and $1\le j\le n$, 
 \[R_{\Theta}M(e_j)=M R_{\Theta'}(e_j)=e^{i \theta_j'} M(e_j)\] 
 for some $\Theta'\in \mathbb R^n$. Thus for each $j$, $M(e_j)$ is an eigen vector for each elements of the torus group, and therefore $M(e_j)=\lambda e_{m_j}$ for some unimodular $\lambda$ and $1\le m_j\le n$. Thus, 
 \[
 M=R_{\widetilde{\Theta}}S
 \]
 for some $R_{\widetilde{\Theta}}\in\mathbb T^n$ and a permutation matrix $S\in \mathcal U(n)$. Consequently, 
 \[
 P(S^{-1}R_{\Theta}S(\bs z))= P(M^{-1}R_{\widetilde{\Theta}}R_{\Theta} R_{\widetilde{\Theta}}^{-1}M(\bs z))= P(\bs z), \quad \bs z\in\mathbb C^n, R_{\Theta}\in\mathbb T^n,
 \]
 where for the last equality we use ~\eqref{invariant for w} and the radial property of $P$. This shows that $P$ is a $S^{-1}\mathbb T^nS$-invariant function. The moreover part follows from $UT=R_{\widetilde{\Theta}}S$, and the fact that $T\in G^{\circ}$, which completes the proof. 
\end{proof}
We are now ready to prove the main result of this section, which will be the key to determine isometric isomorphism of multiplier algebras completely. 
\begin{theorem}\label{thm:existense of permutation}
Let $V, W \subseteq \mathbb B_d$ be monomial varieties with $d'$ generating variables. 
Suppose there exists an automorphism $\Phi$ of $\mathbb B_d$ such that $\Phi(V)=W$.
Then there exists a unitary $S\in \mathcal U(d)$ such that
\[\Phi(\bs z)= S(\bs z), \quad \bs z\in\mathbb B_d.\]
Moreover, there exists a permutation matrix $U\in\mathcal U(d)$ such that $U$ is block diagonal with respect to the decomposition $\mathbb C^d=\mathbb C^{d'}\oplus \mathbb C^{d-d'}$ and  
\[U(V)=W.\] 
\end{theorem}
\begin{proof}
Suppose that the monomial varieties are of the form: 
$$
V = \mathbb B_d \cap \bigcap_{j=d'+1}^d  Z\Big(z_j- B_j \prod\limits_{i=1}^{d'}z_i^{\alpha_{ij}} \Big),
\ \text{ and } 
W = \mathbb B_d \cap \bigcap_{j=d'+1}^d  Z\Big(z_j- C_j  \prod\limits_{i=1}^{d'}z_i^{\beta_{ij}} \Big),
$$
where $\bm\alpha_j=(\alpha_{1j}, \alpha_{2j}, \dots, \alpha_{d'j}),$ $\bm\beta_j=(\beta_{1j}, \beta_{2j}, \dots, \beta_{d'j})\mathbb Z_+^{d'}$ such that both $\bm \alpha_j$'s and $\bm\beta_j$'s are distinct. 
Let us define $f,g:\mathbb C^{d'}\to\mathbb C^{d-d'}$ by
 \[
 f(\bs z)=\big( B_{d'+1} \bm z^{\bm\alpha_{d'+1}},\ldots, 
 B_{d} \bm z^{\bm\alpha_{d}}\big), \quad \bs z\in \mathbb C^{d'}
 \]
 and 
 \[
 g(\bs z)=\big(
 C_{d'+1} \bm z^{\bm\beta_{d'+1}},\ldots, C_{d}\bm z^{\bm\beta_{d}} \big), \quad \bs z\in \mathbb C^{d'}.
 \]
In terms of these functions, the monomial varieties can be described as 
\begin{equation}\label{veriety-domain relation}
V=\{(\bs z, f(\bs z))\in\mathbb B_d: \bs z\in \mathbb B_{d'}\} \text{ and }
W=\{(\bs z, g(\bs z))\in\mathbb B_d: \bs z\in\mathbb B_{d'}\}.
\end{equation}
Now if there is an automorphism $\Phi$ of  $\mathbb B_d$ such that $\Phi(V)=W$, then by Proposition~\ref{prop:unitary-representation} there exists a unitary $M\in\mathcal U(d)$
 \[
 M = \begin{bmatrix}
     A & 0\\
     0 & D
 \end{bmatrix} \text{ such that } \Phi(\bm z)= M(\bm z),\quad \bm z\in \mathbb B_d,
 \] 
for some unitaries $A\in \mathcal U(d')$ and $D\in \mathcal U(d-d')$. 
This proves the first part of the theorem. 

For the moreover part, we continue to work on the unitary $M$. Since $M(V)=W$, by \eqref{veriety-domain relation}, for any $(\bs z, f(\bs z))\in V$ we have  
\[
 M(\bs z, f(\bs z))=(A\bs z, Df(\bs z)) \in W\implies g(A\bs z)=Df(\bs z)~ \text{for all}\ \bs z\in \mathbb B_{d'} \text{ so that  } (\bs z, f(\bs z))\in V.
\]
Since this polynomial identity holds on an open subset of $\mathbb C^{d'}$, by the Identity theorem, it holds on the entire complex plane $\mathbb C^{d'}$, and hence 
 \beq\label{iden_g_f}
 \|g(A\bs z)\|^2=\|f(\bs z)\|^2  \quad \text{for all}~  \bs z\in \mathbb C^{d'}.
 \eeq
Define
 \[
 P(\bs z):=\|f(\bs z)\|^2=\sum_{j=d'+1}^d
 B_j^2 |\bm z^{\bm\alpha_j}|^2,\quad
 \bs z\in \mathbb C^{d'}
 \]
and
\[
Q(\bs z):=\|g(\bs z)\|^2=\sum_{j=d'+1}^d C_j^2|\bm z^{\bm\beta_j}|^2, 
\quad \bs z\in \mathbb C^{d'}.
\]
Then, $P$ and $Q$ are continuous radial function. Moreover, using \eqref{iden_g_f}, for any $R_{\Theta}\in \mathbb T^n$,
 \beqn
 P(A^{-1}R_{\Theta}A(\bs z))=Q(R_{\Theta}A(\bs z))=Q(A(\bs z))=P(\bs z), \quad \bs z\in \mathbb C^{d'}.
 \eeqn
 Then by Lemma~\ref{invariant by permutation}, there exist a permutation matrix $S\in\mathcal U(d')$ and a unitary $T\in \mathcal U(d')$ such that $P$ is an $S^{-1}\mathbb T^{d'}S$-invariant function, $AT=R_{\Theta}S$ for some $R_{\Theta}\in \mathbb T^{d'}$ and $P(T(\bs z))=P(\bs z)$ for all $\bs z$. This in turn implies that, for all $\bs z\in\mathbb C^{d'}$,
 \[
P(\bs z)=P(T(\bs z))= Q(AT(\bs z))=Q(R_{\Theta}S(\bs z))= Q(S(\bs z)).
 \]
Suppose that the permutation matrix $S$ is given by 
\[
 S=\big[e_{\sigma(1)},e_{\sigma(2)},\ldots,e_{\sigma(d')}\big]
 \] for some permutation $\sigma$ of $\{1,2,\dots,d'\}$. Then it is easy to see that 
 \[S(\bm z)= S(z_1,z_2,\dots, z_{d'})=(z_{\sigma^{-1}(1)}, z_{\sigma^{-1}(2)},\dots, z_{\sigma^{-1}(d')})=\sigma^{-1}(\bm z), \quad \bm z\in\mathbb C^{d'}.\]
By replacing the defining identities $P,Q$ in the above we get, 
 \[
 \sum_{j=d'+1}^d
 B_j^2 |\bm z^{\bm\alpha_j}|^2
 = \sum_{j=d'+1}^d C_j^2 |\sigma^{-1}(\bs z)^{\bm\beta_j}|^2= \sum_{j=d'+1}^d C_j^2 |\bs z^{\sigma(\bm\beta_j)}|^2,
 \quad \bs z\in \mathbb C^{d'}.
 \]
Since $\bm\alpha_j$'s and $\bm\beta_j$'s are distinct, the number of terms on both sides of the above identity is the same. Comparing both sides of the above identity as a polynomial of $|z_i|$'s, we infer that it is possible only when 
there exists a permutation $\sigma'$ of $\{d'+1,\ldots,d\}$ such that
\[
B_{\sigma'(j)}^2|\bm z^{\bm\alpha_{\sigma'(j)}}|^2= C_j^2|\bm z^{\sigma(\bm\beta_j)}|^2,
\quad j = d'+1, \ldots, d.
\]
Thus, for each $j = d'+1, \ldots, d$, we have 
 \beq\label{similar_pattern}
 B_{\sigma'(j)}=C_j,\ \text{ and }
 \bm\alpha_{\sigma'(j)}=\sigma(\bm\beta_j).
 \eeq
 This in particular implies for all $\bm z\in\mathbb C^{d'}$
 \begin{align*}
 \sigma'(f(\bm z))& =\big( B_{\sigma'(d'+1)} \bm z^{\bm\alpha_{\sigma'(d'+1)}},\ldots, 
 B_{\sigma'(d)} \bm z^{\bm\alpha_{\sigma'(d)}}\big)\\
 &=\big(
 C_{d'+1} \bm z^{\sigma(\bm\beta_{d'+1})},\ldots, C_{d}\bm z^{\sigma(\bm\beta_{d})} \big)\\
 &= \big(
 C_{d'+1} \sigma^{-1}(\bm z)^{\bm\beta_{d'+1}},\ldots, C_{d}\sigma^{-1}(\bm z)^{\bm\beta_{d}} \big)\\
 &=g(\sigma^{-1}(\bm z)).
 \end{align*}
 Set 
 \[
U:=\big[e_{\sigma(1)},\ldots,e_{\sigma(d')},e_{\sigma'^{-1}(d'+1)},\ldots,e_{\sigma'^{-1}(d)}\big].
 \]
Then, by using the above identity, 
it follows that
 \[
 U(\bm z , f(\bm z))= (\sigma^{-1}(\bm z), \sigma'(f(\bm z)))=(\sigma^{-1}(\bm z), g(\sigma^{-1}(\bm z))), \quad \bm z\in \mathbb C^{d'}. 
 \]
 Thus $U$ is the required block diagonal permutation matrix such that $U(V)=W$. This finishes the proof.
\end{proof}
\begin{remark}

A similar rigidity result remains valid for a slightly more general class of monomial varieties. Namely, one may replace the positive constants $B_j's$ in the form of monomial varieties by arbitrary nonzero complex numbers. No substantial change is required to accommodate this class except that in this generality, one would get a unitary $U$ which is block diagonal such that one diagonal entry is a permutation matrix and the other diagonal entry would be a permutation followed by a phase unitary. In the whole proof of Theorem~\ref{thm:existense of permutation}, the first significant modification occurs in ~\eqref{similar_pattern}.
In this case, one would get 
\[e^{i\theta_j}B_{\sigma'(j)}=C_j, \quad j=d'+1,\dots, d,\]
for some \(\theta_j\in\mathbb R\).
Hence, the unitary matrix \(U\) appearing in the proof can be taken to be
 \[U=\bigl[e_{\sigma(1)},e_{\sigma(2)},\dots,e_{\sigma(d')}, e^{i\theta_{d'+1}}e_{\sigma'^{-1}(d'+1)}, e^{i\theta_{d'+2}}e_{\sigma'^{-1}(d'+2)}, \dots, e^{i\theta_d}e_{\sigma'^{-1}(d)}\bigr],\]
which satisfies \(U(V)=W\). 
\end{remark}

\section{A classification of multiplier algebras up to isometric isomorphism}

In this section, we classify, up to isometric isomorphism, the multiplier algebras associated with integrally generated CNP Dirichlet series kernels and with a class of classical CNP kernels on Reinhardt domains in $\mathbb C^n$. We first establish the classification for Dirichlet series kernels and then transfer it to the classical setting.

\subsection{On CNP Dirichlet series kernels}
A general description of multiplier varieties associated with CNP Dirichlet series kernels was obtained in \cite{ADS}. In the present setting, however, this description can be simplified considerably, yielding a more explicit form. In fact, the result below shows that the multiplier variety associated with a CNP Dirichlet series kernel is always a monomial variety. We first recall the general description obtained in \cite{ADS}.  

Given a frequency data $\mf n$,  $\mathcal J(\mf n)$ denotes the collection of all subsets $J \subseteq \{1, 2, \ldots, d\}$ such that the set $\{\ln n_j: j \in J\}$ is linearly dependent over $\mathbb Q$, but every proper subset is linearly independent over $\mathbb Q$. It was also established in ~\cite[Proposition 4.4]{ADS} that for each $J\in \mathcal J(\mf n)$, there exists a unique tuple $(\beta_j)_{j \in J} \in \mathbb N^{|J|}$ satisfying $\gcd(\beta_j)_{j \in J}=1$, and a unique partition $\{J_1, J_2\}$ of $J$ such that
 \beqn 
 		\prod_{i \in J_1}	n_i^{\beta_i} = \prod_{i \in J_2} n_i^{\beta_i}.
 \eeqn

For each $J\in \mathcal J(\mf n)$, we use these data to define the polynomial   
\beq\label{associated polynomial}
q_{_J}(\bs z) = \prod\limits_{i \in J_1}
 		b_{i}^{\frac{\beta_i}{2}}\prod\limits_{i \in J_2}z_i^{\beta_i}-\prod\limits_{i \in J_2}b_{i}^{\frac{\beta_i}{2}}\prod\limits_{i \in J_1}z_i^{\beta_i},\quad \bs z \in \mathbb C^d.
\eeq
We refer to $q_{_J}$ as {\it the polynomial associated with $J$}. The multiplier varieties can then be described as follows. 
\begin{theorem}{\cite[Theorem 4.6]{ADS}}\label{variety-gen}
Let $(\mf b, \mf n= (n_j)_{j=1}^d)$ be a pair of weight and frequency data of length $d\ge 2$ 
such that $\{\ln n_j: j=1,\dots,d\}$ is linearly dependent over $\mathbb Q$. 
Then, the multiplier variety $V$ associated with $K_{\mf b, \mf n}$ is given by
\beqn
V = \mathbb B_d \cap \bigcap_{J \in \mathcal J(\mf n)} Z(q_{_J}),
 \eeqn
where for each $J\in \mathcal J(\mf n)$, $q_{_J}$ is the polynomial associated with $J$, as in ~\eqref{associated polynomial}. 
 \end{theorem}

We now specialize the preceding description of multiplier varieties to
integrally generated CNP Dirichlet series kernels. In this situation the defining variety
admits a particularly simple description in terms of the generators.
Let
\[
\mf b=(b_j)_{j=1}^d,
\qquad
\mf n=(n_j)_{j=1}^d,
\qquad d\geq 2,
\]
be weight and frequency data. Suppose that, for some \(1\leq d'<d\),
the subtuple $\mf n'=(n_1,\ldots,n_{d'})$ is a generating subtuple of \(\mf n\). That is, \(\{\ln n_1,\ldots,\ln n_{d'}\}\) is linearly independent over $\mathbb Q$, and there exists  $\bm\alpha_j=(\alpha_{1j},\alpha_{2j},\dots,\alpha_{d'j})\in\mathbb Z_+^{d'}$ such that
\begin{equation}\label{representation}
n_j=(\mf n')^{\bm\alpha_j}, \quad j=d'+1,\ldots,d.
\end{equation}
For each \(j=d'+1,\ldots,d\), set
\[
I_j=\{i: 1\le i\le d', \alpha_{ij}\neq 0\}, \text{ and } \widetilde I_j= I_j\cup \{j\}.
\]
To generate the frequency $n_j$ $(d'+1\le j\le d)$, one only needs $\{n_i: i\in I_j\}$. Since the generators \(\ln n_1,\ldots,\ln n_{d'}\) are
independent over $\mathbb Q$, the identity \eqref{representation} implies that
\(
\widetilde I_j\in\mathcal J(\mf n),
\) $\{ I_{j}, \{j\}\}$ is the unique partition, and with $\alpha_{jj}:=1$, $(\alpha_{ij})_{i\in \widetilde{I}_j}$ is the unique tuple of natural numbers associated with $\widetilde{I}_j$.
Then the polynomial corresponding to \(\widetilde I_j\) is
\begin{equation}\label{gen-polynomial}
q_{\widetilde I_j}(\bs z)
=
z_j\prod_{i\in I_j}b_i^{\frac{\alpha_{ij}}{2}}
-
\sqrt{b_j}\prod_{i\in I_j}z_i^{\alpha_{ij}}= z_j(\mf b')^{\frac{1}{2}\bm\alpha_j}-\sqrt{b_j}\prod_{i= 1}^{d'}z_i^{\alpha_{ij}},
\end{equation}
where $\mf b'=(b_1,b_2,\dots,b_{d'})$.
Consequently, Theorem~\ref{variety-gen} gives
\[
V
\subseteq
\mathbb B_d\cap
\bigcap_{j=d'+1}^d Z(q_{\widetilde I_j}).
\]
We next show that this inclusion is an equality. Let
\[
\bs w=(w_1,\ldots,w_d)
\in
\mathbb B_d\cap
\bigcap_{j=d'+1}^d Z(q_{\widetilde I_j}).
\]
By Theorem~\ref{variety-gen}, it suffices to prove that
$q_{_J}(\bs w)=0$ for every $J\in\mathcal J(\mf n)$.
Fix \(J\in\mathcal J(\mf n)\). Let
\(
J=J_1\dot\cup J_2
\)
be the unique partition associated with \(J\), and let
\((\beta_j)_{j\in J}\) be the unique relatively prime positive integers
appearing in the corresponding multiplicative relation
\begin{equation}\label{uniq-rep}
\prod_{j\in J_1}n_j^{\beta_j}
=
\prod_{j\in J_2}n_j^{\beta_j}.
\end{equation}
Define $\bm\beta_1=(\beta_{i1})_{i=1}^d$ and $\bm\beta_2=(\beta_{i2})_{i=1}^d\in \mathbb Z_+^d$ by
\[
\beta_{i1}
=
\begin{cases}
\beta_i,& i\in J_1,\\
0,& i\notin J_1,
\end{cases}
\qquad
\beta_{i2}
=
\begin{cases}
\beta_i,& i\in J_2,\\
0,& i\notin J_2.
\end{cases}
\]
Then \eqref{uniq-rep} is equivalent to
\begin{equation}\label{update_n_i_iden}
\prod_{i=1}^d n_i^{\beta_{i1}}
=
\prod_{i=1}^d n_i^{\beta_{i2}}.
\end{equation}
For \(k=1,2\), using \eqref{representation} we get
\begin{align*}
\prod_{i=1}^d n_i^{\beta_{ik}}
=
\prod_{i=1}^{d'}n_i^{\beta_{ik}}
\prod_{j=d'+1}^d
\left(
\prod_{i=1}^{d'}
n_i^{\alpha_{ij}}
\right)^{\beta_{jk}}
=
\prod_{i=1}^{d'}
n_i^{
\beta_{ik}
+
\sum_{j=d'+1}^d
\alpha_{ij}\beta_{jk}
}.
\end{align*}
Consequently, substituting the above into the both sides of \eqref{update_n_i_iden} yields
\[
\prod_{i=1}^{d'}
n_i^{
\beta_{i1}
+\sum_{j=d'+1}^d
\alpha_{ij}\beta_{j1}
}
=
\prod_{i=1}^{d'}
n_i^{
\beta_{i2}
+\sum_{j=d'+1}^d
\alpha_{ij}\beta_{j2}
}.
\]
Since \(\ln n_1,\ldots,\ln n_{d'}\) are linearly independent over
\(\mathbb Q\), the corresponding exponents must agree. Hence
\begin{equation}\label{iden_of_expon}
\beta_{i1}
+
\sum_{j=d'+1}^d
\alpha_{ij}\beta_{j1}
=
\beta_{i2}
+
\sum_{j=d'+1}^d
\alpha_{ij}\beta_{j2},
\qquad
i=1,\ldots,d'.
\end{equation}
On the other hand, since
\(q_{\widetilde I_j}(\bs w)=0\) for
\(j=d'+1,\ldots,d\), equation \eqref{gen-polynomial} implies
\begin{equation}\label{w_j_rep}
w_j
=
\sqrt{b_j}(\bm b')^{-\frac{1}{2}\bm\alpha_j}
\prod_{i=1}^{d'}w_i^{\alpha_{ij}}=
\sqrt{b_j}
\prod_{i=1}^{d'}
b_i^{-\frac{1}{2}\alpha_{ij}}
w_i^{\alpha_{ij}},
\qquad
j=d'+1,\ldots,d.
\end{equation}
Using \eqref{w_j_rep} and the definition of \(q_{_J}\), as in \eqref{associated polynomial},
 we obtain
\begin{align*}
q_{_J}(\bs w)
&=
\prod_{i=1}^d b_i^{\frac{\beta_{i1}}{2}}
\prod_{i=1}^{d'}w_i^{\beta_{i2}}
\prod_{j=d'+1}^d
\left(
b_j^{\frac{\beta_{j2}}{2}}
\prod_{i=1}^{d'}
b_i^{-\frac{1}{2}\beta_{j2}\alpha_{ij}}
w_i^{\beta_{j2}\alpha_{ij}}
\right)
\\
&\quad \qquad \qquad\qquad-
\prod_{i=1}^d b_i^{\frac{\beta_{i2}}{2}}
\prod_{i=1}^{d'}w_i^{\beta_{i1}}
\prod_{j=d'+1}^d
\left(
b_j^{\frac{\beta_{j1}}{2}}
\prod_{i=1}^{d'}
b_i^{-\frac{1}{2}\beta_{j1}\alpha_{ij}}
w_i^{\beta_{j1}\alpha_{ij}}
\right)
\\
&=
\left(
\prod_{j=d'+1}^d
b_j^{\frac{1}{2}(\beta_{j1}+\beta_{j2})}
\right)
\Bigg[
\prod_{i=1}^{d'}
b_i^{
\frac12
\left(
\beta_{i1}
-
\sum_{j=d'+1}^d
\alpha_{ij}\beta_{j2}
\right)
}
w_i^{
\beta_{i2}
+
\sum_{j=d'+1}^d
\alpha_{ij}\beta_{j2}
}
\\
&\hspace{5cm}
-
\prod_{i=1}^{d'}
b_i^{
\frac12
\left(
\beta_{i2}
-
\sum_{j=d'+1}^d
\alpha_{ij}\beta_{j1}
\right)
}
w_i^{
\beta_{i1}
+
\sum_{j=d'+1}^d
\alpha_{ij}\beta_{j1}
}
\Bigg]\\
&=0 \quad (\text{by } \eqref{iden_of_expon}).
\end{align*}
Since \(J\in\mathcal J(\mf n)\) was arbitrary, Theorem~\ref{variety-gen}
now gives the reverse inclusion. We have therefore proved the following.

\begin{proposition}\label{variety-form}
Let $K_{\mf b, \mf n}$ be an integrally generated CNP Dirichlet series kernel associated with frequency and weight data $(\mf b, \mf n)$ of length $d$, where $\mf n'=(n_1,n_2,\dots, n_{d'})$ is the generating subtuple of $\mf n$ for $1\le d'<d$, and 
\[
n_j= (\mf n')^{\bm \alpha_j},\quad
\bm\alpha_j\in\mathbb Z_+^{d'},
\quad
j=d'+1,\ldots,d.
\]
Then the multiplier variety $V$ associated with \(K_{\mf b,\mf n}\) is a monomial variety with $d'$ generating variables, given by
\beqn
V
=
\mathbb B_d
\cap
\bigcap_{j=d'+1}^d
Z\Big(z_j-\sqrt{b_j} (\bm b')^{-\frac{1}{2}\bm\alpha_j}  \prod\limits_{i=1}^{d'}z_i^{\alpha_{ij}}\Big),
\eeqn
where $\mf b'=(b_1,b_2,\dots, b_{d'})$ and $\bm\alpha_j=(\alpha_{1j}, \alpha_{2j},\dots, \alpha_{d'j})$.

\end{proposition}

The notion of a similar pattern also takes a particularly simple form when the frequency data admits a generating subtuple. This is relevant since integrally generated CNP Dirichlet series kernels have such generating subtuples. We record this observation in the following lemma. It follows directly from the definition that if $(\mf b,\mf n)$ and $(\mf c,\mf m)$ are pairs of weight and frequency data of length $d$, and if $\mf n$ and $\mf m$ admit generating subtuples of different lengths, then $(\mf b,\mf n)$ and $(\mf c,\mf m)$ cannot have a similar pattern. Thus, in studying similar patterns, it is enough to restrict attention to pairs whose frequency data admit generating subtuples of the same length.  

\begin{lemma}\label{lem:similar-pattern}
Let $(\mf b, \mf n)$ and $(\mf c, \mf m)$ be pairs of weight and frequency data of length $d$. Suppose that for $1\le d' < d$, $\mf n'=(n_1, n_2,\dots, n_{d'})$ and $\mf m'=(m_1,m_2,\dots, m_{d'})$ are generating subtuple of $\mf n$ and $\mf m$, respectively, with
    \[
    n_j=  (\mf n')^{\bm\alpha_j}\quad \text{and}\quad m_j=(\mf m')^{\bm\beta_j},\quad \bm \alpha_j, \bm \beta_j\in \mathbb Z_{+}^{d'}, \quad j=d'+1,\dots, d.
    \] 
Then, $(\mf b,\mf n)$ and $(\mf c,\mf m)$ admit a similar pattern if and only if 
\beq\label{eq:less-identities-similar-pattern}
\bm\alpha_j=\bm\beta_j \quad \text{and}\quad b_j (\bm c')^{\bm \beta_j}= c_j(\bm b')^{\bm\alpha_j}, \quad j=d'+1,\dots,d,
\eeq
where $\bm b'=(b_1,b_2,\dots, b_{d'})$ and $\bm c'=(c_1,c_2,\dots, c_{d'})$.
 \end{lemma}

\begin{proof}
Let $V$ and $W$ be the multiplier varieties associated with the integrally generated CNP Dirichlet series kernels $K_{\mf b,\mf n}$ and $K_{\mf c,\mf m}$, respectively. Then, by Proposition~\ref{variety-form}, $V$ and $W$ are monomial varieties and have the form
    \[
    V=\Big\{\bs z\in\mathbb B_d:z_j=\sqrt{b_j} (\bm b')^{-\frac{1}{2}\bm\alpha_j}  \prod\limits_{i=1}^{d'}z_i^{\alpha_{ij}},\ j=d'+1,d'+2,\dots,d\Big\}
    \]
    and 
    \[
    W=\Big\{\bs z\in\mathbb B_d:z_j=\sqrt{c_j} (\bm c')^{-\frac{1}{2}\bm\beta_j}  \prod\limits_{i=1}^{d'}z_i^{\beta_{ij}}\ j=d'+1,d'+2,\dots,d\Big\}.
    \]
By \cite[Proposition~5.5]{ADS}, $(\mf b,\mf n)$ and $(\mf c,\mf m)$ admit a similar pattern if and only if $V=W$. Thus, it is enough to show that $V=W$ if and only if ~\eqref{eq:less-identities-similar-pattern} holds. 
It is easy to see that if the identities in ~\eqref{eq:less-identities-similar-pattern} hold then the defining polynomials of $V$ and $W$ become identical and therefore, $V=W$. Now assume that $V=W$. Recall that the monomial variety $V$ is a graph over the Reinhardt domain 
\[
\widetilde{V}:=\big\{\bs z\in \mathbb C^{d'}:\sum_{j=1}^{d'}|z_j|^2+\sum_{j=d'+1}^db_{j}(\bm b')^{-\bm\alpha_{j}}|\bs z^{2\bm \alpha_j}|<1\big\}.
\]
Then for any $(z_1,z_2,\dots, z_{d'})\in \widetilde{V}$,
    \[
    \big(z_1,\dots,z_{d'},\sqrt{b_j} (\bm b')^{-\frac{1}{2}\bm\alpha_{d'+1}}  \prod\limits_{i=1}^{d'}z_i^{\alpha_{i(d'+1)}},\dots,\sqrt{b_j} (\bm b')^{-\frac{1}{2}\bm\alpha_d}  \prod\limits_{i=1}^{d'}z_i^{\alpha_{id}}\big)\in V=W,
    \]
and therefore by the defining property of $W$, 
\[
\sqrt{b_j} (\bm b')^{-\frac{1}{2}\bm\alpha_j}  \prod\limits_{i=1}^{d'}z_i^{\alpha_{ij}}=\sqrt{c_j} (\bm c')^{-\frac{1}{2}\bm\beta_j}  \prod\limits_{i=1}^{d'}z_i^{\beta_{ij}},\quad j=d'+1,\dots,d.
\]
This shows that all the above polynomial identities hold on an open set $\widetilde{V}$ of $\mathbb C^{d'}$. Hence, by the identity theorem, they hold identically on $\mathbb C^{d'}$, and therefore  
\[
\bm\alpha_j=\bm \beta_j,\ \text{ and } \sqrt{b_j} (\bm b')^{-\frac{1}{2}\bm\alpha_j}=\sqrt{c_j} (\bm c')^{-\frac{1}{2}\bm\beta_j},\quad j=d'+1,\dots,d.
\] 
 This completes the proof. 
\end{proof}

We are now ready for the complete solution to the isometric isomorphism problem of multiplier algebras for integrally generated CNP Dirichlet series kernels. For a permutation $\sigma$ of $\{1,2,\dots,d\}$ and 
$\mf n=(n_1,\dots, n_d)$, recall that we use $\sigma(\mf n)$ to denote $(n_{\sigma(1)}, n_{\sigma(2)},\dots, n_{\sigma(d)})$.
\begin{theorem}\label{rigidity-Isometric}
    Let $K_{\mf b, \mf n}$ and $K_{\mf c, \mf m}$ be integrally generated CNP Dirichlet series kernels associated with the pairs $(\mf b, \mf n)$ and $(\mf c,\mf m)$ of weight and frequency data of length $d$. Then the following are equivalent.
    \begin{itemize}
        
\item [(i)] The multiplier algebras $\clm(\clh(K_{\mf b,\mf n}))$ and $\clm(\clh(K_{\mf c,\mf m}))$ are isometrically isomorphic.
        
\item [(ii)] There exists a permutation matrix $S\in\mathcal U(d)$ such that 
\[S(V)=W,\] where $V$ and $W$ are the multiplier varieties associated with $K_{\mf b, \mf n}$ and $K_{\mf c, \mf m}$, respectively. 
        
\item [(iii)] There exists a permutation $\sigma$ of $\{1,2,\ldots,d\}$ such that $(\sigma (\mf b),\sigma (\mf n))$ and $(\mf c, \mf m)$ admit a similar pattern.
    \end{itemize} 
\end{theorem}
\begin{proof}

$(i) \implies (ii)$. Let $\mathcal M_V$ and $\mathcal M_W$ be the restriction of multiplier algebras of the Drury-Arveson space over $\mathbb B_d$ to the multiplier varieties $V$ and $W$ associated with $K_{\mf b,\mf n}$ and $K_{\mf c,\mf m}$, respectively. 
Since the algebras $\clm(\clh(K_{\mf b,\mf n}))$ and $\clm(\clh(K_{\mf c,\mf m}))$ are isometrically isomorphic, $\mathcal M_V$ and $\mathcal M_W$ are isometrically isomorphic, and hence by \cite[Theorem 5.9]{DRS}, $\mathcal M_V$ and $\mathcal M_W$ are completely isometrically isomorphic. Thus, by \cite[Theorem 4.5]{DRS}, there exists $\Phi \in \text{Aut}(\mathbb B_d)$ such that $\Phi(V) = W$. Observe that, by Proposition~\ref{variety-form},  $V$ and $W$ are monomial varieties. Since they are biholomorphic, the number of generating variables for $V$ and $W$ must be the same. Otherwise, the corresponding Reinhardt domains $\widetilde V$ and $\widetilde W$, as defined in \eqref{graph-domain}, would lie in complex Euclidean spaces of different dimensions, contradicting the fact that they are biholomorphic. Finally, since $V$ and $W$ are monomial varieties with same number of generating variables and $\Phi(V)=W$ for an automorphism $\Phi$ of $\mathbb B_d$, (ii) follows from Theorem~\ref{thm:existense of permutation}. 

$(ii) \implies (iii)$. Since $V$ and $W$ are biholomorphic monomial varieties, the same argument used above in the proof of $(i)\Rightarrow(ii)$ shows that they have the same number of generating variables. Moreover, by Proposition~\ref{variety-form}, the number of generating variables of a multiplier variety agrees with the length of the generating subtuple of the corresponding frequency data. Hence, the generating subtuples of $\mf n$ and $\mf m$ have the same length, say $d'$. Suppose that $\mf n'=(n_j)_{j=1}^{d'}$ and $\mf m'=(m_j)_{j=1}^{d'}$ are the generating subtuples of $\mf n$ and $\mf m$, respectively, and
\[
    n_j=  (\mf n')^{\bm\alpha_j}\ \text{and } m_j=(\mf m')^{\bm\beta_j},\quad \bm \alpha_j, \bm \beta_j\in \mathbb Z_{+}^{d'}, \quad j=d'+1,\dots, d.
\] 
Then by Proposition~\ref{variety-form}, 
 \[
 V=\Big\{\bs z\in\mathbb B_d:z_j=\sqrt{b_j} (\bm b')^{-\frac{1}{2}\bm\alpha_j}  \prod\limits_{i=1}^{d'}z_i^{\alpha_{ij}},\ j=d'+1,d'+2,\dots,d\Big\}
 \]
    and 
 \[
 W=\Big\{\bs z\in\mathbb B_d:z_j=\sqrt{c_j} (\bm c')^{-\frac{1}{2}\bm\beta_j}  \prod\limits_{i=1}^{d'}z_i^{\beta_{ij}}\ j=d'+1,d'+2,\dots,d\Big\},
\]
where $\mf b'=(b_1, b_2,\dots, b_{d'})$ and $\mf c'=(c_1, c_2,\dots,c_{d'})$. 
Now, by (ii) and Theorem~\ref{thm:existense of permutation}, we assume that the permutation matrix $U\in\mathcal U(d)$ for which $U(V)=W$ is of the form 
\[
U = \Big[e_{\sigma(1)},\dots,e_{\sigma(d')},e_{\sigma'(d'+1)}, \ldots, e_{\sigma'(d)}\Big]
\]
for some permutations $\sigma$ of $\{1,2,\dots,d'\}$ and $\sigma'$ of $\{d'+1,\dots,d\}$. Then note that 
\begin{equation}\label{definition of U}
U(z_1,z_2,\dots,z_d)=(z_{\sigma^{-1}(1)},\dots,z_{\sigma^{-1}(d')},z_{\sigma'^{-1}(d'+1)},\dots,z_{\sigma'^{-1}(d)}). 
\end{equation}
For any $(z_1,z_2,\dots, z_{d'})\in \widetilde{V}$, the Reinhardt domain associated with $V$ as in ~\eqref{graph-domain}, 
\[
\big(z_1,\dots,z_{d'},\sqrt{b_j} (\bm b')^{-\frac{1}{2}\bm\alpha_{d'+1}}  \prod\limits_{i=1}^{d'}z_i^{\alpha_{i(d'+1)}},\dots,\sqrt{b_j} (\bm b')^{-\frac{1}{2}\bm\alpha_d}  \prod\limits_{i=1}^{d'}z_i^{\alpha_{id}}\big)\in V,
\]
and therefore
\begin{align*}
&U\big(z_1,\dots,z_{d'},\sqrt{b_{d'+1}} (\bm b')^{-\frac{1}{2}\bm\alpha_{d'+1}}  \prod\limits_{i=1}^{d'}z_i^{\alpha_{i(d'+1)}},\dots,\sqrt{b_d} (\bm b')^{-\frac{1}{2}\bm\alpha_d}  \prod\limits_{i=1}^{d'}z_i^{\alpha_{id}}\big)\in W.
\end{align*}
Then by a routine calculation using ~\eqref{definition of U} and the description of $W$, we get 
\[
\sqrt{b_{\sigma'^{-1}(j)}} (\bm b')^{-\frac{1}{2}\bm\alpha_{\sigma'^{-1}(j)}}  \prod\limits_{i=1}^{d'}z_i^{\alpha_{i\sigma'^{-1}(j)}}=\sqrt{c_j} (\bm c')^{-\frac{1}{2}\bm\beta_j}  \prod\limits_{i=1}^{d'}z_{\sigma^{-1}(i)}^{\beta_{ij}},\quad j=d'+1,\dots,d,
 \]
for all $(z_1,z_2,\dots,z_{d'})\in \widetilde{V}$. Applying the Identity theorem we obtain
\[
\bm\alpha_{\sigma'^{-1}(j)}=\sigma(\bm \beta_j)\quad \text{and}\quad \sqrt{b_{\sigma'^{-1}(j)}} (\bm b')^{-\frac{1}{2}\bm\alpha_{\sigma'^{-1}(j)}}=\sqrt{c_j} (\bm c')^{-\frac{1}{2}\bm\beta_j}, \quad j=d'+1,\dots,d.
\]
By Lemma~\ref{lem:similar-pattern}, this shows that
$(\tilde{\sigma}^{-1}(\mf b),\tilde{\sigma}^{-1}(\mf n))$ and $(\mf c,\mf m)$ admit a similar pattern, where $\tilde{\sigma}$ is the permutation of $\{1,2,\dots,d\}$ obtained by joining $\sigma$ and $\sigma'$ together. This establishes (iii).

$(iii)\implies(i)$. 
It is easy to see that $K_{\mf b, \mf n}= K_{\sigma(\mf b), \sigma(\mf n)}$ for any permutation $\sigma$ of $\{1,2,\dots, d\}$. Then (i) follows from \cite[Theorem 5.6]{ADS}. This finishes the proof. 
\end{proof}
\begin{remark}
(a) As stated in Theorem~\ref{thm:existense of permutation}, the permutation automorphism appearing in item~{\rm (ii)} of the preceding result may be chosen to be block diagonal with respect to the decomposition
\[
\mathbb C^d=\mathbb C^{d'}\oplus \mathbb C^{d-d'}.
\]
Equivalently, the permutation \(\sigma\) of \(\{1,2,\ldots,d\}\) appearing in item~{\rm (iii)} may be chosen so that it is induced by a permutation of \(\{1,2,\ldots,d'\}\) together with a permutation of \(\{d'+1,d'+2,\ldots,d\}\).

(b) The equivalence between {\rm (i)} and {\rm (iii)} reveals the rigidity of the weight and frequency data imposed by an isometric isomorphism between the multiplier algebras associated with integrally generated CNP Dirichlet series kernels.

(c) The equivalence between {\rm (i)} and {\rm (ii)} admits a useful rigidity interpretation. Namely, if \(V,W\subseteq \mathbb B_d\) are monomial varieties, then \[ \mathcal M_V \overset{\mathrm{isometric}}{\cong}\mathcal M_W \] if and only if there exists a permutation matrix $U\in\mathcal U(d)$
such that \(U(V)=W. \) Thus, within the class of monomial varieties, the isometric isomorphism class of the multiplier algebra completely determines the underlying variety up to permutation of the ambient coordinates. This may be viewed as a multiplier-algebra analogue of Gelfand-Kolmogorov theorem for commutative \(C^*\)-algebras. 
However, the conclusion is substantially more rigid. An isometric isomorphism between the multiplier algebras of two monomial varieties does not merely recover the varieties up to an abstract homeomorphism or biholomorphism. Rather, it recovers them up to a permutation of the ambient coordinate axes. Consequently, the normed algebra \(\mathcal M_V\) retains precise information about the way in which \(V\) is embedded in the ambient ball.
\end{remark}
We now illustrate the theorem with two examples: one yielding isometric isomorphism and the other yielding algebras that are not isometric isomorphism.  
\begin{example}
(i) Consider the frequency tuples 
\[
\mf n=(p_1,p_2,p_1^2,p_1p_2,p_1^3,p_2^2,p_1^2p_2,p_1^3p_2^2),\ \text{ and }
    \mf m=(q_1,q_2,q_1q_2,q_1^2,q_1^3q_2^2,q_2^2,q_1^2q_2,q_1^3)
\]
such that both $\{\ln p_1,\ln p_2\}$ and $\{\ln q_1,\ln q_2\}$ are linearly independent over $\mathbb Q$. For strictly positive real numbers $b_1,b_2,c_1,c_2$, define the weight tuples
\[
    \mf b=(b_1,b_2,b_1^2,b_1b_2,b_1^3,b_2^2,b_1^2b_2,b_1^3b_2^2),\  \text{ and }
    \mf c=(c_1,c_2,c_1c_2,c_1^2,c_1^3c_2^2,c_2^2,c_1^2c_2,c_1^3).
\]
Consider the permutation $\sigma=(3\ 4)(5\ 8)$, with all other indices fixed. Then observe that after applying $\sigma$, the entries of $\mf n$ have precisely the same monomial expressions in $p_1,p_2$ as the corresponding entries of $\mf m$ have in $q_1,q_2$. The same applies to the weight data $\mf b$ and $\mf c$ also. Consequently, $(\sigma(\mf b),\sigma(\mf n))$ and $(\mf c,\mf m)$ admit a similar pattern. By Theorem~\ref{rigidity-Isometric},
    \[
    \clm(\clh(K_{\mf b,\mf n}))
    \overset{\mathrm{isometric}}{\cong}
    \clm(\clh(K_{\mf c,\mf m})).
    \]

(ii) Consider the frequency tuples 
\[
    \mf n=(p_1,p_2,p_1^2,p_1p_2,p_1^3,p_2^2,p_1^2p_2,p_1^3p_2^2),\ \text{ and }
    \mf m=(q_1,q_2,q_1^2,q_1q_2,q_1^4,q_2^2,q_1^2q_2,q_1^3q_2^2),
    \]
such that $\{\ln p_1,\ln p_2\}$ and $\{\ln q_1,\ln q_2\}$ are linearly independent over $\mathbb Q$. The tuple $\mf m$ contains both $q_1$ and $q_1^4$, whereas the linear independence of $\ln p_1$ and $\ln p_2$ ensures that no entry of $\mf n$ is the fourth power of another.
Thus, no permutation of $\mf n$ can preserve the multiplicative relations of $\mf m$. It follows that for any strictly positive weight tuples $\mf b,\mf c$ and any permutation $\sigma$ of $\{1,\ldots,8\}$, the pairs $(\sigma(\mf b),\sigma(\mf n))$ and $(\mf c,\mf m)$ cannot admit a similar pattern. Hence, by Theorem~\ref{rigidity-Isometric}, the multiplier algebras
    \[
    \clm(\clh(K_{\mf b,\mf n}))
    \quad\text{and}\quad
    \clm(\clh(K_{\mf c,\mf m}))
    \]
    are not isometrically isomorphic, regardless of the choice of weights.

\end{example}

\subsection{On classical CNP kernels}
There is a well-known connection between CNP Dirichlet series kernels and classical CNP kernels on domains in \(\mathbb C^n\) that preserves the associated multiplier varieties. This connection is established in \cite[Section~7]{ADS}. We use it to characterize isometric isomorphisms between multiplier algebras for a certain class of CNP kernels defined on domains in \(\mathbb C^n\).

The class of CNP kernels we consider is determined by a pair $(\mf b, \bm \alpha)$, where
\[\mf b = (b_j)_{j=1}^d \in (0, \infty)^d \text{ and } \bm \alpha=\{\bm\alpha_{d'+1}, \bm\alpha_{d'+2}, \dots, \bm\alpha_d\} \subseteq \mathbb Z_+^{d'},\]
for some $ d\ge 2$ and $d'<d$ such that  $\bm \alpha_j=(\alpha_{1j}, \alpha_{2j},\dots,\alpha_{d'j})$ and $\bm\alpha_j$'s are disitnct with $|\bm\alpha_j|\ge 2$. Given such a pair, define 
\[
\Omega_{\mf b, \bm \alpha}:=\Big\{\bs z\in \mathbb C^{d'}:\sum\limits_{j=1}^{d'}b_j |z_j|^2 +\sum\limits_{j=d'+1}^d b_j \prod\limits_{i=1}^{d'} |z_i|^{2\alpha_{ij}}<1\Big\},
\]
and consider the CNP kernel  $\cls_{\mf b, \bm \alpha} : \Omega_{\mf b, \bm \alpha} \times \Omega_{\mf b, \bm \alpha} \rar \mathbb C$ given by 
\begin{equation}
\label{sbm}
\cls_{\mf b, \bm \alpha} (\bs z, \bs  w) 
= \frac{1}{1 - 
\sum\limits_{j=1}^{d'}b_j z_j \overline{w_j} - \sum\limits_{j=d'+1}^d b_j \prod\limits_{i=1}^{d'} (z_i \overline{w_i})^{\alpha_{ij}}} = \frac{1}{1 - \inp{f({\bs  z})}{f({\bs w})}_{\mathbb C^d}}
\end{equation}
where the embedding map $f$ is given by
\begin{equation}\label{eq:def-f_b_alpha}
    f(\bs z) = (\sqrt{b_{1}}z_1, \ldots, \sqrt{b_{d'}}z_{d'}, \sqrt{b_{d'+1}}{\bs z}^{\bm \alpha_{d'+1}}, \ldots, \sqrt{b_{d}}{\bs z}^{\bm \alpha_{d}}), \quad \bm z=(z_1,\dots, z_{d'}),
\end{equation}
for all $\bs z, \bs w\in \Omega_{\mf b, \bm \alpha}$.
Corresponding to each of these kernels, we can associate a CNP Dirichlet series kernel without changing the multiplier varieties as follows. Fix $d'$ distinct primes and enumerate them as $\{p_i:i=1,\dots,d'\}$. For any $\bm\gamma = (\gamma_1, \gamma_2, \ldots, \gamma_{d'}) \in \mathbb Z_+^{d'}$, we set 
\begin{align*}
\mf p^{\bm\gamma}: = \prod_{i =1}^{d'} p_i^{\gamma_i}.
\end{align*}
Then, the CNP Dirichlet series kernel associated to $\cls_{\mf b, \bm \alpha}$ is given by 
\begin{align*}
K_{\mf b, \mf n}(s, u) = \frac{1}{1 - \sum\limits_{j=1}^{d'}b_j p_j^{-s-\overline{u}} - \sum\limits_{j=d'+1}^d b_j \prod\limits_{i=1}^{d'} p_i^{-\alpha_{ij}(s+\overline{u})}}, \quad s, u \in \mathbb H_\rho,
\end{align*}
where $\bm\alpha_j := (\alpha_{1j},  \dots, \alpha_{d'j}) \in \mathbb Z_+^{d'}$ for each $j=d'+1, \ldots, d$, $\mf b = (b_{1}, \ldots, b_{d})$ is the weight data, $\mf n = (p_1, \ldots, p_{d'}, \mf p^{\bm\alpha_{d'+1}}, \ldots, \mf p^{\bm\alpha_d})$ is the frequency data, and $\rho\in\mathbb R$ is chosen so that the kernel is well-defined on $\mathbb H_\rho$. 
By \cite[Proposition~7.1]{ADS}, the multiplier varieties associated with \(\cls_{\mf b,\bm\alpha}\) and \(K_{\mf b,\mf n}\) coincide. Therefore, applying the characterization obtained for CNP Dirichlet series kernels in Theorem~\ref{rigidity-Isometric}, we obtain the following result.
We note that the special case \(d'=d-1\) of the following theorem was obtained in \cite[Theorem~7.3]{ADS}.

\begin{theorem}\label{rigidity-isometric-classical}
   Let $\cls_{\mf b, \bm \alpha}$ and $\cls_{\mf c, \bm \beta}$ be CNP kernels as in ~\eqref{sbm}. Then the following are equivalent: 
\begin{itemize}
\item[(i)] $\mathcal{M}(\mathcal{H}(\cls_{\mf{b}, \bm \alpha}))$ and $\mathcal{M}(\mathcal{H}(\cls_{\mf{c}, \bm \beta}))$ are isometrically isomorphic.
\item[(ii)] There exist permutations $\sigma$ of $\{1,2,\ldots, d'\}$ and $\sigma'$ of $\{d'+1,d'+2,\ldots,d\}$ such that 
\begin{equation}\label{similarity}
\sigma(\bm \alpha_{\sigma'(j)})=\bm\beta_j \ \text{ and } (\mf c')^{\bm\beta_j}b_{\sigma'(j)}
=
\sigma(\mf b')^{\bm\beta_j}c_j,\quad  j=d'+1,\dots, d,
\end{equation}
where $
\bm b'=(b_1,b_2,\dots, b_{d'})\ \text{ and } \bm c'=(c_1, c_2,\dots, c_{d'})$.

\item[(iii)]  There exists a permutation matrix $U\in\mathcal U(d)$ such that
\[
U(f(\Omega_{\mf b, \bm \alpha}))=g(\Omega_{\mf c, \bm \beta}),
\]
where $f$ and $g$ are embedding maps 
as defined in \eqref{eq:def-f_b_alpha}.
\end{itemize} 
\end{theorem}
\begin{proof} 
Suppose $K_{\mf b, \mf n}$ and $K_{\mf c, \mf m}$ are the CNP Dirichlet series kernels associated with $\cls_{\mf b, \bm \alpha}$ and $\cls_{\mf c, \bm \beta}$, respectively, where
$$
\mf n=\big(p_1, \dots, p_{d'}, \mf p^{\bm\alpha_{d'+1}}, \ldots, \mf p^{\bm\alpha_d}\big)~\text{and}~ 
\mf m=\big(p_1,p_2,\dots, p_{d'}, \mf p^{\bm\beta_{d'+1}}, \ldots, \mf p^{\bm\beta_d}\big),$$
for distinct primes $p_1,\dots, p_{d'}$. Since the above association does not change the multiplier varieties, 
\[
\mathcal{M}(\mathcal{H}(\cls_{\mf{b}, \bm\alpha})) \overset{\mathrm{isometric}}{\cong} \mathcal{M}(\mathcal{H}(\cls_{\mf{c}, \bm \beta}))\ \text{ if and only if } \mathcal{M}(\mathcal{H}(K_{\mf{b}, \mf n}))\overset{\mathrm{isometric}}{\cong} \mathcal{M}(\mathcal{H}(K_{\mf{c}, \mf m})).
\] 
$(i) \implies (ii)$. Since $\mathcal{M}(\mathcal{H}(K_{\mf{b}, \mf n}))$ is isometrically isomorphic to $\mathcal{M}(\mathcal{H}(K_{\mf{c}, \mf m}))$, by Theorem~\ref{rigidity-Isometric},
we obtain a permutation \(\tilde{\sigma}\) of \(\{1,2,\ldots,d\}\) such that \((\tilde{\sigma}(\mf b),\tilde{\sigma}(\mf n))\) and \((\mf c,\mf m)\) admit a similar pattern. Since the first $d'$ many frequencies of $\mf m$ is a generating subtuple for $\mf m$, the same must be true for $\tilde{\sigma}(\mf n)$ also. This forces $\tilde{\sigma}$ to have the property that 
\[\tilde{\sigma}(\{1,2,\dots, d'\})=\{1,2,\dots, d'\}.\] 
This allow us to define permutations $\sigma$ of the sets \(\{1,\ldots,d'\}\) and  $\sigma'$ of the set \(\{d'+1,\ldots,d\}\) as 
\[
\sigma(i)=\tilde{\sigma}(i),\quad i=1,\ldots,d'\quad\text{and}\quad \sigma'(i)=\tilde{\sigma}(i),\qquad i=d'+1,\ldots,d.
\] 
Since
\[
\tilde{\sigma}(\mf b)=(b_{\sigma(1)}, b_{\sigma(2)},\dots, b_{\sigma(d')}, b_{\sigma'(d'+1)},\dots, b_{\sigma'(d)})\]
and 
\[
\tilde{\sigma}(\mf n)=(p_{\sigma(1)}, p_{\sigma(2)},\dots, p_{\sigma(d')}, \mf p^{\bm \alpha_{\sigma'(d'+1)}},\dots, \mf p^{\bm \alpha_{\sigma'(d)}}),
\]
the generating subtuple for $\tilde{\sigma}(\mf n)$ is $\sigma(\mf p)=(p_{\sigma(1)}, p_{\sigma(2)},\dots, p_{\sigma(d')})$. Moreover, the $j$-th frequency of $\tilde{\sigma}(\mf n)$ can be expressed as 
\[
\mf p^{\bm \alpha_{\sigma'(j)}}=\sigma(\mf p)^{\sigma (\bm \alpha_{\sigma'(j)})}, \quad j=d'+1,\dots,d.
\]
Then by applying Lemma~\ref{lem:similar-pattern} to $(\tilde{\sigma}(\mf b), \tilde{\sigma}(\mf n))$ and $(\mf c, \mf m)$ we get 
\[
\sigma(\bm\alpha_{\sigma'(j)})=\bm \beta_j\ \text{ and } (\bm c')^{\bm\beta_j}b_{\sigma'(j)}= \sigma(\bm b')^{\sigma(\bm\alpha_{\sigma'(j)})} c_j,\quad  j=d'+1,\dots, d,
\]
where $\mf b'=(b_1, b_2,\dots, b_{d'})$ and $\mf c'=(c_1, c_2,\dots, c_{d'})$.
This proves (ii).



$(ii) \implies (iii)$. Define a permutation matrix $U\in \mathcal U(d)$ by
\[
U=\Big[e_{\sigma^{-1}(1)},\dots,e_{\sigma^{-1}(d')},e_{\sigma'^{-1}(d'+1)},\dots,e_{\sigma'^{-1}(d)}\Big].
\] 
Then for all $(z_1,z_2,\dots, z_d)\in \mathbb C^d$, $U$ satisfies
\[
U(z_1,z_2,\dots,z_d)=\big(z_{\sigma(1)},\dots,z_{\sigma(d')},z_{\sigma'(d'+1)},\dots,z_{\sigma'(d)}\big).
\]
Recall that for all $\bm z\in \Omega_{\mf b, \bm \alpha}$,
\[
f(\bm z)=(\sqrt{b_{1}}z_1, \ldots, \sqrt{b_{d'}}z_{d'}, \sqrt{b_{d'+1}}{\bs z}^{\bm \alpha_{d'+1}}, \ldots, \sqrt{b_{d}}{\bs z}^{\bm \alpha_{d}}).
\]
Let $\bm z\in \Omega_{\mf b, \bm\alpha}$ be fixed. Then 
\[
U(f(\bs z))=\big(\sqrt{b_{\sigma(1)}}z_{\sigma(1)},\dots,\sqrt{b_{\sigma(d')}}z_{\sigma(d')},\sqrt{b_{\sigma'(d'+1)}}\bs z^{\alpha_{\sigma'(d'+1)}},\dots,\sqrt{b_{\sigma'(d)}}\bs z^{\alpha_{\sigma'(d)}}\big).
\]
Set
\[
w_i:=\frac{\sqrt{b_{\sigma(i)}}}{\sqrt{c_i}}z_{\sigma(i)},\quad 1\Le i\Le d'.
\]
We show that $\bs w\in\Omega_{\mf c, \bm \beta}$ and $g(\bs w)=U(f(\bs z))$, where $\bs w=(w_1,w_2,\dots,w_{d'})$. To this end, first observe that for any $d'+1\le j\le d$,
\[
\sqrt{c_j}\bs w^{\bm \beta_j}=\sqrt{c_j}\frac{\sigma(\mf b')^{\frac{1}{2}\bm\beta_j}}{(\mf c')^{\frac{1}{2}\bm\beta_j}} \sigma(\bm z)^{\bm\beta_j} \overset{\eqref{similarity}}= \sqrt{c_j}\frac{\sqrt{b_{\sigma'(j)}}}{\sqrt{c_j}} \bm z^{\bm\alpha_{\sigma'(j)}}=\sqrt{b_{\sigma'(j)}}\bm z^{\bm\alpha_{\sigma'(j)}}. 
\]
This shows that $U(f(\bs z))=g(\bs w) $. Since $\bs z\in \Omega_{\mf b, \bm \alpha}$, we have $\|f(\bs z)\|<1$. Then $\|g(\bs w)\|= \|U(f(\bs z))\|=\|f(\bs z)\|<1$, and hence $\bs w\in \Omega_{\mf c, \bm \beta}$. Since $\bm z\in \Omega_{\mf b, \bm \alpha}$ was arbitrary,  
\[U(f(\Omega_{\mf b, \bm \alpha}))\subseteq g(\Omega_{\mf c, \bm \beta}).\]
The reverse inequality is obtained in a similar way by considering $U^{-1}$ and proving that 
\[U^{-1}(g(\Omega_{\mf c, \bm \beta}))\subseteq f(\Omega_{\mf b, \bm \alpha}).
\]
This proves (iii).

$(iii) \implies (i)$. Suppose that there exists a permutation matrix $U\in \mathcal U(d)$ such that
\[
U(f(\Omega_{\mf b, \bm \alpha}))=g(\Omega_{\mf c, \bm \beta}).
\]
Let $V$ and $W$ be multiplier varieties corresponding to the kernels $\cls_{\mf b, \bm \alpha}$ and $\cls_{\mf c, \bm \beta}$, respectively. That is \[
V=\{\bs z\in\mathbb B_d: h(\bs z)=0, \text{ whenever } h\in \mathcal M_d\ \text{ and } h|_{f(\Omega_{\mf b, \bm \alpha})}\equiv 0\},
\]
and $W$ is defined similarly. First, we show that $U(V)=W$. To this end, by the definition of $V$, for any $\bs z\in \mathbb B_d\setminus V$ there exists $h\in \mathcal M_d$ such that $h(\bs z)\neq 0$ and $h|_{f(\Omega_{\mf b, \bm \alpha})}\equiv 0$. Then $h\circ U^{-1}\in \mathcal M_d$ and 
\[
h\circ U^{-1}|_{g(\Omega_{\bm c, \bm\beta})}=
h\circ U^{-1}|_{U(f(\Omega_{\bm b, \bm\alpha}))}\equiv 0, \text{ however } h\circ U^{-1}(U(\bs z))= h(\bs z)\neq 0.
\]
It follows that $U(\bs z)\in \mathbb B_d\setminus W$ whenever $\bs z\in \mathbb B_d\setminus V$. Thus,
\[
W\subseteq U(V)
\]
For the other inclusion, let $\bs z\in V$. For any $h\in\clm_d$ with $h\circ U|_{f(\Omega_{\mf b, \bm \alpha})}\equiv0$, since $V$ is the smallest multiplier variety containing $f(\Omega_{\mf b, \bm \alpha})$ and $h\circ U\in\clm_d$, we get 
\[
h\circ U|_{V}\equiv 0, \text{ in particular } h(U(\bs z))=0.
\]
This shows that $h(U(\bs z))=0$ whenever $h\in \mathcal M_d$ and $h|_{g(\Omega_{\bm c, \bm\beta})}\equiv 0$. Consequently, by the definition of $W$, $U(\bs z)\in W$, and since $\bs z$ was arbitrary
\[
U(V)\subseteq W.
\]
Thus, $U(V)=W$. Now, applying \cite[Theorem 4.5]{DRS}, we conclude that $\clm_V$ and $\clm_W$ are isometrically isomorphic, and hence $\mathcal{M}(\mathcal{H}(\cls_{\mf{b}, \alpha}))$ and $\mathcal{M}(\mathcal{H}(\cls_{\mf{c}, \bm \beta}))$ are isometrically isomorphic.
This completes the proof.
\end{proof}

\vspace{0.1in} 

\noindent\textbf{Acknowledgement:} The first named author is supported by Prime Minister's Research Fellowship, ID: 1303115. The second named author's research is supported by ANRF-MATRICS Grant
ANRF/ARGM/2025/002455/MTR. The third named author acknowledges the support received from the Indian Institute of Technology Bombay through an Institute postdoctoral fellowship.

\vspace{0.1in}
\noindent\textbf{Data Availability:}
Data sharing is not applicable to this article, as no datasets were generated or analyzed during
the current study. In case any datasets are generated during and/or analyzed during the current study, they must
be available from the corresponding author on reasonable request.

\vspace{0.1in}
\noindent\textbf{Conflict of interest:} There is no conflict of interest.

\end{document}